\documentclass[11pt]{article}
\usepackage{amsmath, amsthm, amssymb, amscd, mathptmx}

\usepackage{amsfonts}

\newcommand\imCMsym[4][\mathord]{%
	\DeclareFontFamily{U} {#2}{}
	\DeclareFontShape{U}{#2}{m}{n}{
		<-6> #25
		<6-7> #26
		<7-8> #27
		<8-9> #28
		<9-10> #29
		<10-12> #210
		<12-> #212}{}
	\DeclareSymbolFont{CM#2} {U} {#2}{m}{n}
	\DeclareMathSymbol{#4}{#1}{CM#2}{#3}
}

\imCMsym{cmmi}{124}{\Cjmath}

\usepackage{color}
\usepackage{pstricks}

\usepackage[all]{xy}\CompileMatrices\SelectTips{cm}{12}

\theoremstyle{plain}
\newtheorem{Thm}{\sc Theorem}[section]
\newtheorem{Theorem}[Thm]{\sc Theorem}
\newtheorem{Corollary}[Thm]{\sc Corollary}

\newtheorem*{Corollary*}{\sc Corollary}

\newtheorem{Proposition}[Thm]{\sc Proposition}
\newtheorem*{Proposition*}{\sc Proposition}
\newtheorem{Lemma}[Thm]{\sc Lemma}

\theoremstyle{definition}
\newtheorem{Definition}[Thm]{Definition}

\theoremstyle{remark}
\newtheorem{Remark}[Thm]{Remark}
\newtheorem{Example}[Thm]{Example}
\newtheorem*{Example*}{Example}
\newtheorem*{Remark*}{Remark}

\renewcommand{\AA}{{\mathbb A}}

\newcommand{\CC}{{\mathbb C}}

\newcommand{\PP}{{\mathbb P}}

\newcommand{\cA}{{\mathcal A}}

\newcommand{\cB}{{\mathcal B}}

\newcommand{\cE}{{\mathcal E}}
\newcommand{\cF}{{\mathcal F}}
\newcommand{\cG}{{\mathcal G}}

\newcommand{\cI}{{\mathcal I}}
\newcommand{\cK}{{\mathcal K}}
\newcommand{\cL}{{\mathcal L}}

\newcommand{\cM}{{\mathcal M}}
\newcommand{\cN}{{\mathcal N}}
\newcommand{\cO}{{\mathcal O}}

\newcommand{\cQ}{{\mathcal Q}}

\newcommand{\cT}{{\mathcal T}}
\newcommand{\cV}{{\mathcal V}}
\newcommand{\cW}{{\mathcal W}}

\newcommand{\codim}{{\mathop{\rm codim \, }}}

\newcommand{\GL}{\mathop{\rm GL\, }}

\newcommand{\an}{{\mathop{\rm an }}}

\newcommand{\ch}{{\mathop{\rm ch \, }}}

\newcommand{\rk}{{\mathop{\rm rk \,}}}

\newcommand{\Sym}{{\mathop{{\rm Sym \, }}}}

\newcommand{\Supp}{{\mathop{{\rm Supp \,}}}}
\newcommand{\Spec}{{\mathop{{\rm Spec\, }}}}

\newcommand{\Vect}{{\mathop{{\rm Vect \,}}}}

\newcommand{\wf}{{\mathop{\rm wf}}}

\def\MR#1{}

\begin{document}

\markboth {\rm }{}

\title{Weak positivity and weak flatness of vector bundles}
\author{Adrian Langer} \date{\today}

\maketitle


{\noindent \sc Address:}\\
Institute of Mathematics, University of Warsaw,
ul.\ Banacha 2, 02-097 Warszawa, Poland\\
e-mail: {\tt alan@mimuw.edu.pl}

\medskip

\begin{abstract}
	We study Viehweg's weak positivity for vector bundles on
	quasi-projective schemes over noetherian rings, including mixed
	characteristic. We prove that weak positivity is preserved under
	tensor products, symmetric powers, divided powers, and exterior
	powers. We introduce weakly flat bundles, generalizing numerically
	flat bundles, and we use them to construct an S-fundamental group
	scheme for normal varieties admitting a small projective
	compactification. We compare weak flatness with strong numerical
	flatness and establish analogues of the Demailly–Peternell–Schneider theorem. 
	For smooth complex varieties admitting a small compactification, we prove that the semisimple objects
	in the category of weakly flat bundles are precisely the unitary
	flat bundles. We also show that a vector bundle equipped with an
	integrable algebraic connection and an invariant filtration with
	unitary flat quotients is weakly flat. As applications, we characterize quotients of abelian varieties in
	terms of weak positivity of some standard vector bundles.
\end{abstract}

\section*{Introduction}

Viehweg introduced weak positivity of vector bundles over a dense
open subset of a quasi-projective variety in \cite{Vi1} and developed
its basic properties in \cite{Vi2,Vi3}. This notion has become an
important tool in the study of moduli spaces of polarized varieties
and positivity on non-projective varieties.
One aim of this paper is to establish preservation of weak
positivity over arbitrary noetherian bases under tensor operations, extending the
characteristic-zero treatment in \cite[Section~2.3]{Vi2}.
Our first main result is the following.

\begin{Theorem}\label{main1}
	Let $X$ be a quasi-projective scheme over a noetherian ring $R$,
	and let $X_0\subseteq X$ be a fixed Zariski dense open subset.
	Then the class of vector bundles weakly positive over $X_0$ is
	closed under finite direct sums, quotient vector bundles, tensor
	products, positive symmetric powers, positive divided powers, and
	exterior powers of positive rank.
\end{Theorem}

Corollary~\ref{operations:weak-positivity} proves a stronger statement
for bundles associated with polynomial representations and establishes
the analogous properties for ampleness with respect to
$X_0$. Closure under finite direct sums and quotients is elementary;
the main assertions concern tensor products, divided powers, and
exterior powers. The theorem applies in mixed characteristic and
requires an argument over the base, beyond the corresponding
statements on geometric fibers. Its proof uses the author's recent
results \cite{La-Ample} on tensor operations on ample vector bundles
over possibly non-proper schemes.

We also study positivity in positive characteristic. In particular,
we characterize weak positivity over a fixed dense open subset
in terms of ampleness of Frobenius pullbacks with respect to that
open subset (see Proposition~\ref{frobenius-wp}).

A second aim of the paper is to extend aspects of the theory of
numerically flat bundles to non-projective varieties.
Recall that a vector bundle $\cE$ is numerically flat if both
$\cE$ and $\cE^*$ are nef.
Demailly, Peternell and Schneider proved that a vector bundle on a
smooth complex projective variety is numerically flat if and only
if it admits a filtration by subbundles with unitary flat quotients
\cite[Theorem~1.18]{DPS1994}.
Moreover, Simpson's results show that such a bundle admits an
integrable algebraic connection preserving a filtration of this
kind and inducing unitary connections on its quotients
(see \cite[Corollary~3.10 and the discussion following it]{Simpson1992}).

On non-projective varieties, numerical flatness in the curve-theoretic
sense is too weak for a direct extension of these results
(see \cite[Example 3.2]{La-Simpson} and Example \ref{ex:unitary-filtration-not-weakly-positive}).
This motivates replacing nefness by weak positivity over the whole
variety, which agrees with nefness in the projective case.
We call a vector bundle $\cE$ on $X$ \emph{weakly flat} if both
$\cE$ and $\cE^*$ are weakly positive over $X$.

Theorem~\ref{main1} shows that weakly flat vector bundles form a rigid
symmetric monoidal category $\Vect^{\wf}(X)$.
If $X$ is a big open subset of a normal projective variety, we prove
that this category, equipped with evaluation at a point $x\in X(k)$,
is neutral Tannakian (see Proposition~\ref{Vect-wf-rigid-abelian}).
Its Tannaka dual defines the S-fundamental group scheme
$\pi_1^S(X,x)$.
In positive characteristic, this agrees with the S-fundamental group
scheme introduced in \cite[Definition~4.18]{La-Simpson}.

We compare weak flatness with strong numerical flatness, introduced
in \cite{La-Simpson}, and study extension across the boundary of a
small compactification. The following theorem summarizes the
comparison when the reflexive extension is locally free or the
compactification is smooth.

\begin{Theorem}\label{main2}
	Let $\overline X$ be a normal projective variety of dimension
	$n\ge2$ over an algebraically closed field $k$.
	Fix an ample line bundle $H$ on $\overline X$.
	Let $j:X\hookrightarrow\overline X$ be a big open subset, let
	$\cE$ be a vector bundle on $X$, and set $\cF:=j_*\cE$.
	Assume that either $\overline X$ is smooth or $\cF$ is locally
	free on $\overline X$.
	Then the following conditions are equivalent:
	\begin{enumerate}
		\item $\cE$ is strongly numerically flat.
		\item $\cE$ is weakly flat.
		\item $\cF$ is strongly slope $H$-semistable and
		\[
		\int_{\overline X}\ch_1(\cF)H^{n-1}
		=
		\int_{\overline X}\ch_2(\cF)H^{n-2}
		=0.
		\]
		\item $\cF$ is a numerically flat vector bundle on
		$\overline X$.
	\end{enumerate}
	Moreover, $(2)$ implies
	$(1)$ without either additional assumption on $\overline X$
	or $\cF$.
\end{Theorem}

The equivalence is proved in
Theorem~\ref{nf-normal:prop-weak-positivity}.
In positive characteristic, conditions $(1)$--$(3)$ are equivalent
without either additional assumption; see
Theorem~\ref{snf-wp:equivalence}.
We also prove that in arbitrary characteristic  weak flatness implies strong numerical flatness
on every normal quasi-projective variety, even without a small
compactification (see
Corollary~\ref{cor:weak-flat-implies-snf-char-zero}).

We use the above theorem, together with the recent result of Ejiri and Yoshikawa \cite[Theorem 1.2]{Ejiri-Yoshikawa}, to 
 obtain the following characterization of quotients of abelian varieties in positive characteristic (see Corollary \ref{cor:frobenius-weak-positivity-abelian}).

\begin{Corollary}\label{cor:main2}
Let $X$ be a smooth projective globally $F$-split variety
of positive dimension over an algebraically closed field
$k$ of characteristic $p>0$, and let $U\subseteq X$ be
a big open subset. Then the following conditions are
equivalent:
\begin{enumerate}
  \item There exist an ordinary abelian variety $A$
    and a finite surjective \'etale morphism
    $\pi:A\to X$.  
    \item The vector bundle $\Omega^1_{X/k}$ is
    weakly positive over $U$.
  \item For some integer $e\geq1$, 
    \(
    (F_X^e)_*\cO_X
    \)
    is weakly positive over $U$.
    \item For every integer $e\geq1$,  $(F_X^e)_*\cO_X$ is numerically flat.
\end{enumerate}
\end{Corollary}

For general normal varieties in characteristic zero, a formulation of (3) in Theorem \ref{main2}  is more delicate.
When $X$ is smooth and $\overline X$ has klt singularities, we prove
a stronger version using orbifold Chern classes and a finite
quasi-\'etale cover (see
Theorem~\ref{snf-wp:equivalence-klt}). This generalizes an earlier result of Lu and Taji \cite{LT}. Together with the results of Greb, Kebekus and Peternell \cite{GKP-projectively-flat} we obtain the following characterization of quotients of abelian varieties (see Corollary \ref{cor:normalized-cotangent-weak-positivity}).

\begin{Theorem}\label{main4}
	Let $X$ be a complex projective klt variety of dimension
$n\ge2$. Then the following conditions are equivalent:
\begin{enumerate}
	\item The vector bundle
	\(	\Sym^n\Omega_{X_{\mathrm{reg}}}^1\otimes\omega_{X_{\mathrm{reg}}}^{-1}	\)
	is weakly positive over $X_{\mathrm{reg}}$.
	\item There exist an abelian variety $A$ and a finite
	surjective quasi-\'etale morphism $A\to X$. 
\end{enumerate}
\end{Theorem}

If $(\Sym^n\Omega_{X}^1\otimes\omega_{X}^{-1})^{**}$ is nef in the sense of \cite{GKP-projectively-flat} then
(1) holds. Thus Theorem \ref{main4} strengthens \cite[Theorem 1.3]{GKP-projectively-flat} by replacing nefness of the normalized reflexive cotangent sheaf on \(X\) with weak positivity of its restriction over \(X_{\mathrm{reg}}\). Example \ref{ex:normalized-cotangent-not-nef} shows that this weakening is strict.

\medskip

Finally, we obtain the following analogue of the
Demailly--Peternell--Schneider theorem (see Propositions~\ref{prop:unitary-uniform-generation}
and~\ref{prop:stable-unitary-weak-positivity}).

\begin{Theorem}\label{main3}
	Let $X$ be a smooth complex variety admitting a small normal
	projective compactification $X\subset\overline X$.
	Let $\cE$ be a vector bundle on $X$.
	\begin{enumerate}
		\item Assume that $\cE$ is equipped with an integrable
		algebraic connection $\nabla$ and a filtration
		\[
		0=\cE_0\subset\cE_1\subset\cdots\subset\cE_s=\cE
		\]
		by $\nabla$-invariant subbundles.
		If the induced connections on the quotients
		$\cE_i/\cE_{i-1}$ have unitary monodromy, then $\cE$
		is weakly flat.
		\item The bundle $\cE$ is unitary flat if and only if it is
		a semisimple object of $\Vect^{\wf}(X)$.
	\end{enumerate}
\end{Theorem}

The proof of (2) in Theorem \ref{main3} depends crucially on the recent flatness criterion 
of Cao, Deng and Matsumura (see \cite{CDM-flatness}). 
Note that every weakly flat vector
bundle admits a filtration by weakly flat subbundles with unitary
flat quotients. We do not know whether it always admits an integrable algebraic
connection preserving such a filtration and inducing unitary
connections on the quotients.
The existence of the underlying filtration alone does not imply
weak flatness, unlike in the projective case (see
Example~\ref{ex:unitary-filtration-not-weakly-positive}).
Thus the connection hypothesis in $(1)$ cannot simply be omitted.

\medskip

A substantial part of this paper originates in an unpublished
manuscript written by the author in 2023. 
In particular,  the results of Section 4
(including Proposition~\ref{frobenius-wp}, which plays a crucial role in
the proof of Theorem~\ref{snf-wp:equivalence}) were already known
to the author at that time. However, Theorem~\ref{main1},
whose proof relies on \cite{La-Ample}, was then out of reach,
which delayed publication. Although the author had already proved
Theorem~\ref{snf-wp:equivalence-klt}, most of the results of
Section~6.2 were not known to the author at that time. A small
part of the discussion of the relationship between strongly
numerically flat and weakly flat bundles revises and strengthens
material that appeared in Section~4.2 of version~2 of
\cite{La-Simpson} but was omitted from the subsequent version.

\medskip

The paper is organized as follows. In Section 1 we recall the definition of strong numerical flatness
and prove a few auxiliary results on reflexive sheaves on normal projective surfaces.
In Section 2  we develop weak positivity over
noetherian bases. In Section 3 we prove preservation of weak positivity under tensor operations
(in particular, we  prove Theorem \ref{main1}).  Section 4 is devoted to the study of  Frobenius criteria in positive
characteristic.
In Section 5 we  construct the S-fundamental group scheme
and prove most of Theorem \ref{main2} and Corollary \ref{cor:main2}. In the final section we prove the remaining part of Theorem \ref{main2} and we prove Theorem \ref{main3}. We also study weak flatness on klt compactifications in characteristic zero and prove Theorem \ref{main4}.

\medskip

\subsection*{Notation}

A vector bundle on a locally noetherian scheme $X$ is a locally
free $\cO_X$-module of finite rank.
It is called \emph{ample} if $\cO_{\PP_X(\cE)}(1)$ is ample.

For a coherent sheaf $\cF$ on $X$ we write
\[
\Sym^{[m]}\cF:=(\Sym^m\cF)^{**}\quad \hbox{and} \quad\Gamma^{[m]}\cF=(\Sym^m(\cF^*))^*.
\]
If $X$ is integral and $\cF$ has generic rank $r$, we set
\(\det\cF:=(\textstyle\bigwedge^r\cF)^{**}. \)

In  positive characteristic we denote the absolute Frobenius morphism
by $F_X$ and set
\[
F_X^{[e]}\cF:=((F_X^e)^*\cF)^{**}.
\]

Let $X$ be a  scheme over an algebraically
closed field $k$.
A vector bundle $\cE$ on $X$ is \emph{nef} if, for every morphism
$f:C\to X$ from a smooth connected projective $k$-curve, every
quotient vector bundle of $f^*\cE$ has non-negative degree.
It is \emph{numerically flat} if both $\cE$ and $\cE^*$ are nef. 
Over an arbitrary field, nefness and numerical flatness are understood 
geometrically, after extension to an algebraic closure.

An open subset $U$ of a normal variety $X$ is \emph{big} if
$\codim(X\setminus U)\ge2$.
A normal variety $X$ is \emph{normally big} if it admits an open
immersion $j:X\hookrightarrow\overline X$ as a big open subset
of a normal projective variety.
We call $\overline X$ a \emph{small compactification} of $X$.

In characteristic zero, strong slope semistability means slope
semistability.

\section{Preliminaries}

For reflexive sheaves on normal projective surfaces (or normal projective varieties in positive characteristic), Chern classes and intersection numbers are understood in the sense of
\cite{La-Inters} (see also \cite{Langer2000} for the characteristic zero version).

Let us recall the following less standard definition from \cite{La-Simpson}.

\begin{Definition}\label{Def:strongly-num-flat} 
	Let $X$ be a normal variety defined over an algebraically closed field $k$.
	A vector bundle $\cE$ on $X$ is
	\emph{strongly numerically flat} if it is numerically flat and, 
	for every normally big surface $S$, its small compactification
	$i:S\hookrightarrow\overline S$, and every morphism $g:S\to X$,
	\[
	\int_{\overline S}\ch_2(i_*g^*\cE)=0.
	\]
\end{Definition}

From now on in this section $S$ is a normal projective surface over an algebraically closed
field. The following lemma can be easily derived from the author's results (see also \cite[Lemma 2.1]{La-Bog-normal}
for the positive characteristic case).
 
 \begin{Lemma}\label{lem:surface-chern-subadditivity}
 	Let 
 	\[
 	0\longrightarrow\cF_1\longrightarrow\cF
 	\longrightarrow\cF_2
 	\]
 	be a left exact sequence of reflexive sheaves which is
 	right exact on a big open subset. Then
 	\[
 	\int_S\ch_2(\cF)
 	\le
 	\int_S\ch_2(\cF_1)+\int_S\ch_2(\cF_2).
 	\]
 \end{Lemma}
 
 \begin{proof}
 	Choose a projective resolution $\pi:T\to S$ dominating
 	simultaneous flattenings of $\cF$ and
 	$\operatorname{im}(\cF\to\cF_2)$.
 	The resulting quotient of vector bundles gives an exact
 	sequence
 	\[
 	0\longrightarrow\cE_1\longrightarrow\cE
 	\longrightarrow\cE_2\longrightarrow0
 	\]
 	on $T$, with
 	$
 	(\pi_*\cE)^{**}\simeq\cF$ and
 	$(\pi_*\cE_i)^{**}\simeq\cF_i$ for $i=1,2$.	
 	For each exceptional fiber, the definition of the relative
 	second Chern class gives
 	\[
 	c_2(\pi_x,\cE)
 	\le
 	c_2(\pi_x,\cE_1)+c_2(\pi_x,\cE_2)
 	+c_1(\pi_x,\cE_1)c_1(\pi_x,\cE_2).
 	\]
 This can be seen directly from \cite[Definition~3.1]{La-Inters} by  pulling back approximate splitting filtrations of
 	$\cE_1$ and $\cE_2$ to a common generically finite cover 	and concatenating them using the pulled-back exact sequence.
 	Additivity of the first relative Chern class therefore gives
 	\[
 	\ch_2(\pi_x,\cE)
 	\ge
 	\ch_2(\pi_x,\cE_1)+\ch_2(\pi_x,\cE_2).
 	\]
 	Subtracting these local corrections from the ordinary
 	Chern characters on $T$, which are additive, proves
 	the assertion (see \cite[Definition~4.1]{La-Inters}).
 \end{proof}

\begin{Lemma}\label{snf-wp:symmetric-chern-bound}
	Let $\cE$ be a reflexive 
	sheaf of rank $r>0$. Then 
	\begin{equation*}
		\chi(S,\Sym^{[n]}\cE)
		\le
		\frac{n^{r+1}}{(r+1)!}
		\left(c_1(\cE)^2-\int  _S c_2(\cE)\right)+O(n^r).
	\end{equation*}
\end{Lemma} 

\begin{proof}
The inequality follows from \cite[Theorem 4.4 and Theorem 3.12]{La-Inters} (cf. \cite[Theorem 4.15]{Langer2000}).	
\end{proof}

\medskip

\begin{Lemma}\label{nf-normal:lem-surface-char-zero}
	Let $\cE$ be a vector bundle on $S$.
	Suppose that $\cE$ is strongly slope $H$-semistable for some ample divisor $H$
	and that
	\[
	c_1(\cE)\cdot H=0,
	\qquad 
	\int_S\ch_2(\cE)=0.
	\]
	Then $\cE$ is numerically flat.
\end{Lemma}
\begin{proof}
	We may assume that $r:=\rk\cE>0$.
	Bogomolov's inequality and the Hodge index theorem give
	\[
	0\leq \int_S\Delta(\cE)
	=c_1(\cE)^2-2r\int_S\ch_2(\cE)
	=c_1(\cE)^2\leq0.
	\]
	Consequently, equality in the Hodge index theorem implies
	that $c_1(\cE)\equiv0$. Hence $c_2(\cE)$ is also numerically
	trivial.
	
	Suppose first that $\operatorname{char}k=p>0$.
	The bundles $(F_S^e)^*\cE$, for $e\ge0$, are slope
	$H$-semistable and have numerically trivial Chern classes.
	By Riemann--Roch, they have the same Hilbert polynomial.
	They therefore form a bounded family by
	\cite[Theorem~4.2]{La-AnnMath}.
	The boundedness criterion
	\cite[Theorem~2.9]{FL} implies that $\cE$ is numerically flat.
	
	Assume now that $\operatorname{char}k=0$.
	Let $f:T\to S$ be a projective resolution, and put
	$\cE_T=f^*\cE$ and $L=f^*H$.
	Then $L$ is nef and big, $\cE_T$ is slope $L$-semistable,
	and
	\[
	c_1(\cE_T)\equiv0,
	\qquad
	\int_T\ch_2(\cE_T)=0.
	\]
	We show that $\cE_T$ is semistable with respect to an ample
	polarization on $T$.
	
	Let $\cG\subset\cE_T$ be a nonzero proper saturated
	subsheaf with $c_1(\cG)\cdot L=0$, and put
	$\cQ=\cE_T/\cG$.
	Both $\cG$ and $\cQ$ are slope $L$-semistable of slope zero.
	Bogomolov's inequality for nef polarizations
	\cite[Theorem~3.2]{La-AnnMath} and the Hodge index theorem give
	\[
	\int_T\ch_2(\cG)
	\leq \frac{c_1(\cG)^2}{2\rk\cG}\leq0,
	\qquad
	\int_T\ch_2(\cQ)
	\leq \frac{c_1(\cQ)^2}{2\rk\cQ}\leq0.
	\]
	Since the two Chern characters add to
	$\int_T\ch_2(\cE_T)=0$, all these inequalities are equalities.
	Thus $c_1(\cG)^2=0$, and the Hodge index theorem yields
	$c_1(\cG)\equiv0$.
	
	Fix an ample divisor $A$ on $T$ and choose $m>0$ and an
	embedding
	\[
	\cE_T\hookrightarrow\cO_T(mA)^{\oplus N}.
	\]
	Then for every nonzero proper saturated subsheaf
	$\cG\subset\cE_T$, we have
	$\mu_A (\cG)\leq mA^2.$
	If $\mu_L(\cG)<0$, then 
	$\mu_L(\cG)\leq-1/r$. Hence, for any rational number
	$0<\varepsilon<1/(rmA^2)$,
	\[
	c_1(\cG)\cdot(L+\varepsilon A)<0.
	\]
	If $c_1(\cG)\cdot L=0$, the preceding argument gives
	$c_1(\cG)\equiv0$.
	Therefore $\cE_T$ is slope semistable with respect to the
	ample rational divisor $L+\varepsilon A$.
So $\cE$ satisfies the semistability hypothesis of \cite[Theorem~2.13]{FL}, which implies that
	$\cE$ is numerically flat.
\end{proof}

\medskip

The following lemma generalizes Kleiman's result \cite[Theorem 3]{Kleiman1969}.

\begin{Lemma}\label{snf-wp:ample-surface-segre}
	Let  $U\subset S$ be a
	big open subset. Let $\cE$ be a reflexive sheaf of rank $r>0$ such
	that $\cE|_U$ is a vector bundle. If $\cE|_U$ is ample, then
	\begin{equation*}
	s_2(\cE):=c_1(\cE)^2-	\int_S c_2(\cE)>0.
	\end{equation*}
\end{Lemma}

\begin{proof}
	Put $P=\PP_U(\cE|_U)$, let $\pi:P\to U$ be the projection, and
	set $L=\cO_P(1)$. Since $L$ is ample, there exist an integer $c>0$
	and an immersion	$	h:P\hookrightarrow\PP_k^N$ such that $h^*\cO_{\PP_k^N}(1)\simeq L^c.$
	Let $Y$ be the reduced closure of the image of $(\pi,h)$ in
	$S\times\PP_k^N$, and write
	\[
	q:Y\to S,\qquad
	M=\operatorname{pr}_2^*\cO_{\PP_k^N}(1)|_Y.
	\]
	The graph morphism over $U$ is a closed immersion, since $\pi$ is
	proper. Hence $q^{-1}(U)=P$ and $M|_P=L^c$.
	The variety $Y$ is integral and projective of dimension $r+1$, and
	$M$ is $q$-very ample. The second projection is generically an
	isomorphism onto its image, because its restriction to the dense
	open subset $P$ is $h$. In particular,
	\[
	M^{r+1}>0.
	\]
	
	For $t\gg0$, relative Serre vanishing gives $R^iq_*M^t=0$ for all
	$i>0$. The sheaf $q_*M^t$ is torsion free and restricts to
	$\Sym^{ct}(\cE|_U)$ on $U$. Its reflexive hull is therefore
	$\Sym^{[ct]}\cE$, so there is an injection
	\[
	q_*M^t\hookrightarrow\Sym^{[ct]}\cE
	\]
	with zero-dimensional cokernel. It follows that
$$	\chi(S,\Sym^{[ct]}\cE)
		\ge\chi(S,q_*M^t)=\chi(Y,M^t)=\frac{M^{r+1}}{(r+1)!}t^{r+1}+O(t^r).$$
	Comparing this with the inequality from Lemma \ref{snf-wp:symmetric-chern-bound}, applied with
	$n=ct$, yields
	\[
	c^{r+1} s_2(\cE)\ge M^{r+1}>0.
	\]
\end{proof}

\section{Weakly positive sheaves}

Let $R$ be a noetherian ring and let $X$ be a quasi-projective  $R$-scheme. 
In \cite{Vi1} and \cite[2.3]{Vi2} E. Viehweg introduced and studied the following definition over an algebraically closed field.

We use ordinary ampleness: a vector bundle $\cE$ is ample if
$\cO_{\PP_X(\cE)}(1)$ is ample on $\PP_X(\cE)$. Since the base
$\Spec R$ is affine, this is equivalent to ampleness relative to
$\Spec R$. 

\begin{Definition}
	Let $X_0\subset X$ be a Zariski dense open subset.
	We say that a vector bundle $\cE$ is 
	\begin{enumerate}
		\item \emph{globally generated over $X_0$}  if the evaluation map $\Gamma(X, \cE)\otimes \cO_X\to \cE$ is surjective over $X_0$.
		\item 	\emph{weakly positive over $X_0$} if for every ample line bundle $H$ on $X$ and 	every positive integer $a$ there exists  a positive integer $b $ such that $\Sym ^{ab}\cE \otimes _{\cO_X}H^{ b}$ is globally generated over $X_0$.
		\item  \emph{weakly positive} if there exists  a Zariski dense open subset $X_0\subset X$ such that $\cE$ is weakly positive over $X_0$.
	\end{enumerate}
\end{Definition}

One can also introduce another variant of weak positivity but this is in fact a special case of the above.

\begin{Definition}
	Assume that $X$ is normal.
	Let $\cE$ be a coherent $\cO_X$-module and let $V_{\cE}$ be the largest open subset of $X$
	on which $\cE':=\cE/\mathrm{torsion} $ is locally free. Let $X_0\subset V_{\cE}$ be a Zariski dense open subset.
	We say that $\cE$ is \emph{weakly positive over $X_0$} if for every ample line bundle $H$ on $X$ and 	every positive integer $a$ there exists  a positive integer $b $ such that $\Sym ^{[ab]}\cE \otimes _{\cO_X}H^{ b}$ is globally generated over $X_0$.
	We say that $\cE$ is \emph{weakly positive} if there exists  a Zariski dense open subset $X_0\subset V_{\cE}$ such that $\cE'$ is weakly positive over $X_0$.
\end{Definition}

Note that $\cE$ is weakly positive over $X_0$ if and only if $\cE'|_{V_{\cE}}$ is weakly positive over $X_0$ in the previous sense.

Generation over an open subset always means
generation by sections defined on $X$, as in
\cite[Definitions 2.10 and 2.11]{Vi2}. The definition of weak positivity
is independent of the ample line bundle, by
\cite[Lemma 2.14(a)]{Vi2}; its proof also works over a noetherian ring, without a reducedness assumption.
Indeed, to replace $H$ by another ample line bundle $A$, choose $c>0$
such that $A^c\otimes H^{-1}$ is globally generated and apply the
condition for $H$ with $ac$ in place of $a$.

If $X$ is normal, $j:X_0\hookrightarrow X$ is big and $\cE$ is locally free on $X_0$, we have
\[
j_*\Sym^m\cE=\Sym^{[m]}(j_*\cE).
\]
Consequently, weak positivity of $\cE$ over all of $X_0$ agrees with
weak positivity of $j_*\cE$ over $X_0$, computed on $X$.

Weak positivity over $X$ was introduced to replace nefness on quasi-projective varieties, which is weaker than weak positivity over $X$.
Let us recall the following lemma from \cite[Lemma 2.24]{Vi2}. The reducedness assumption in the cited statement is unnecessary; see Lemma~\ref{open-ample-equivalence} below. 

\begin{Lemma}\label{Viehweg-ampleness}
	Let $X$ be a quasi-projective $R$-scheme.
	Let $\cE$ be a vector bundle and $H$ an ample line bundle on $X$. Then the following conditions are equivalent:
	\begin{enumerate}
		\item $\cE$ is ample.
		\item For some $d>0$, $\Sym ^d\cE\otimes _{\cO_X} H^{-1}$ is globally generated over $X$.
		\item  For some $d>0$, $\Sym ^d\cE\otimes _{\cO_X} H^{-1}$ is weakly positive over $X$.
	\end{enumerate}
\end{Lemma}

The above lemma motivated Viehweg to introduce another definition of ampleness with respect to an open subset
(see \cite[Definition 1.2(c)]{Vi3}). The corresponding equivalence over an open subset is explicitly included in that definition. We give the details, also without a reducedness assumption, in Lemma~\ref{open-ample-equivalence} below.

\begin{Definition}\label{ample-over-open}
	A vector bundle $\cE$ on $X$ is \emph{ample with respect to $X_0$} if, for some ample line bundle $H$ and for some
	$d>0$, the bundle
	\(	\Sym^d\cE\otimes H^{-1}\)
	is globally generated over $X_0$.
\end{Definition} 

By the following lemma, the above definition agrees with
\cite[Definition 1.2(c)]{Vi3} and it is independent of $H$.

\begin{Lemma}\label{open-ample-equivalence}
	Fix an ample line bundle $H$ on $X$. For a vector bundle $\cE$ on $X$, the following conditions are equivalent:
	\begin{enumerate}
		\item $\cE$ is ample with respect to $X_0$.
		\item For some $d>0$, $\Sym^d\cE\otimes H^{-1}$ is weakly positive 
		over $X_0$.
		\item For every coherent $\cO_X$-module $\cF$, the sheaf
		$\Sym^m\cE\otimes\cF$ is globally generated over $X_0$ for $m\gg0$.
	\end{enumerate}
	In particular, this notion is independent of $H$, and for $X_0=X$
	it is equivalent to ordinary ampleness.
\end{Lemma}

\begin{proof}
	Assume $(1)$. Then for some ample line bundle $A$ and for some
	$d>0$, the bundle
	\(	\Sym^d\cE\otimes A^{-1}\)
	is globally generated over $X_0$. Write $m=dn+i$, where $0\le i<d$. The natural multiplication map 
	\[
	\Sym^n(\Sym^d\cE\otimes A^{-1})
	\otimes(\Sym^i\cE\otimes\cF\otimes A^n)
	\longrightarrow\Sym^m\cE\otimes\cF
	\]
	is	surjective.
	The first factor is generated over $X_0$, and the second is globally
	generated on $X$ for $n\gg0$, uniformly for the finitely many possible
	values of $i$. This proves $(3)$.  $(3)\Rightarrow(1)$ follows by
	taking $\cF=H^{-1}$. Since $(3)$ does not involve $H$, it also
	proves independence of $H$. For $X_0=X$, condition $(3)$ is the usual definition of ampleness.
	
	A sheaf generated over $X_0$ is weakly positive over $X_0$, so
	$(1)\Rightarrow(2)$. Conversely, under $(2)$ choose $b>0$ such that
	\[
	\Sym^{2b}(\Sym^d\cE\otimes H^{-1})\otimes H^b
	\]
	is generated over $X_0$. Its quotient
	$\Sym^{2db}\cE\otimes H^{-b}$ is also generated over $X_0$.
	We may replace $b$ by a sufficiently large multiple so that
	$H^{b-1}$ is globally generated. Tensoring with this line bundle
	proves $(1)$.
\end{proof}

\medskip

The following lemma is well-known but we recall it for the convenience of the reader.

\begin{Lemma}\label{snf-wp:elementary-properties}
	Let $R$ be a noetherian ring, let $X$ be a quasi-projective
	$R$-scheme, and let $X_0\subset X$ be a Zariski dense open subset.
	\begin{enumerate}
		\item Weak positivity of vector bundles over the entire scheme is
		preserved by arbitrary pullback between quasi-projective $R$-schemes.
		More generally, if $f:Z\to X$ is a morphism of quasi-projective
		$R$-schemes and $f^{-1}(X_0)$ is dense in $Z$, then $f^*\cE$ is
		weakly positive over $f^{-1}(X_0)$ whenever $\cE$ is weakly positive over $X_0$.
		\item If a vector bundle on $X$ is weakly positive over $X_0$,
		its determinant is weakly positive over the same subset.
		\item Assume that $R=k$ is a field, let $\cE$ be locally free on $X$ and
		weakly positive over $X_0$, and let $f:C\to X$ be a morphism from
		a smooth projective integral curve over a field extension of $k$ whose image meets $X_0$. Then $f^*\cE$
		is nef.
	\end{enumerate}
	In particular, if $R=k$ is a field and both $\cE$ and $\det\cE^*$ are weakly positive
	over the entire scheme $X$, then $\cE$ is numerically flat.
\end{Lemma}

\begin{proof}
	For (1), choose ample line bundles $H$ on $X$ and $A$ on $Z$.
	Choose $c>0$ such that $A^c\otimes f^*H^{-1}$ is globally generated.
	For every $a>0$, choose $b>0$ such that
	$\Sym^{acb}\cE\otimes H^b$ is generated over $X_0$.
	Pulling back and tensoring with $(A^c\otimes f^*H^{-1})^b$ shows that
	$\Sym^{acb}(f^*\cE)\otimes A^{cb}$ is generated over $f^{-1}(X_0)$.
	This proves (1), extending \cite[Lemma 2.15]{Vi2}.
	(2) follows from Corollary \ref{operations:weak-positivity} (see also \cite[Lemma 1.4, 6)]{Vi1} for the case when $R$ is a field).
	
	To prove (3), fix an ample line bundle $H$ on $X$.
	For every positive integer $a$ there exists a positive integer $b $ such that
	$\Sym ^{ab}(f^*\cE) \otimes (f^*H)^{ b}$ is generically globally generated.
	So for any finite surjective morphism $g:C'\to C$ from a smooth projective integral curve and any quotient line bundle $g^*(f^*\cE)\to L$ we have
	\[
	ab\deg L+b\deg (g^*f^*H)\ge 0.
	\]
	Dividing by $ab$ and letting $a$ tend to infinity gives $\deg L\ge 0$, proving
	nefness of $f^*\cE$.
	
	Finally, under the last hypothesis, (3) shows that every pullback
	of $\cE$ to a smooth projective curve is nef, while its inverse
	determinant is also nef. So every such pullback has degree zero and hence $\cE$ is numerically flat.
\end{proof}

For the remainder of this section, let $k$ be an algebraically closed field
and let $X$ be a normal projective $k$-variety of dimension $n$. The following lemma will be needed in the proof of Proposition \ref{prop:comparison-of-fund-groups}.

\begin{Lemma}\label{lem:numerical-criterion} 
Let $B$ be a Weil divisor on $X$ and
let $H$ be an ample line bundle such that $	B\cdot H^{n-1}=0.$
Let	$f:Y\to X$ be a projective birational morphism from
smooth $Y$. If $\cO_X(B)$ is weakly positive then the numerical class of any line bundle extending
$f^*\cO_X(B)$ from the inverse image of a big open subset of $X$
belongs to the rational span of the $f$-exceptional divisors.
\end{Lemma}
\begin{proof}
Let $\widetilde B$ denote the strict transform of $B$ on $Y$.
Any line bundle as in the statement differs from
$\cO_Y(\widetilde B)$ by an $f$-exceptional divisor. It therefore
suffices to prove the assertion for $[\widetilde B]$.

Fix an ample divisor $A$ on $Y$. By
\cite[Corollary~3.16]{FL-positive}, there is a norm on
$N^1(Y)_{\mathbb R}$ whose restriction to effective divisor
classes is given by intersection with $A^{n-1}$.
Moreover, \cite[Proposition~3.21]{FL-positive}, together with
the same corollary on $X$, gives a constant $C>0$ such that
\[
\|[\widetilde D]\|
=\widetilde D\cdot A^{n-1}
\le C\,D\cdot H^{n-1}
\]
for every effective Weil divisor $D$ on $X$.
Indeed, the proposition provides an effective lift $D'$ with
$f_*D'=D$ and the required degree bound, and
$D'-\widetilde D$ is effective and $f$-exceptional.

By weak positivity, for every integer $a>0$ there exist an
integer $b_a>0$ and an effective Weil divisor $D_a$ such that
\[
D_a\sim ab_aB+b_aH.
\]
Since $B\cdot H^{n-1}=0$, the preceding estimate yields
\[
\left\|\frac{[\widetilde D_a]}{ab_a}\right\|
\le
\frac{C\,D_a\cdot H^{n-1}}{ab_a}
=
\frac{C\,H^n}{a}.
\]
Thus $[\widetilde D_a]/(ab_a)$ converges to zero in
$N^1(Y)_{\mathbb R}$.

Let $V\subset N^1(Y)_{\mathbb R}$ be the real span of the
classes of the $f$-exceptional prime divisors.
The linear equivalence defining $D_a$ implies
\[
\frac{[\widetilde D_a]}{ab_a}
-[\widetilde B]-\frac1a[f^*H]\in V.
\]
Passing to the limit and using that $V$ is closed, we obtain
$[\widetilde B]\in V$.
Finally, $[\widetilde B]$ and the exceptional divisor classes
belong to $N^1(Y)_{\mathbb Q}$, so rational linear algebra
shows that $[\widetilde B]$ belongs to their rational span.
\end{proof}

\begin{Lemma}\label{surface-weak-positivity}
Assume that $n=2$. Let $X_0\subset X$ be a big open subset
	and let $\cE$ be a vector bundle on $X$. Then $\cE$ is weakly
	positive over $X_0$ if and only if $\cE$ is nef on $X$.
\end{Lemma}

\begin{proof}
	A nef vector bundle on a projective variety is weakly positive
	over the whole variety (see \cite[Remarks~2.12]{Vi2}).
	Conversely, every nonconstant curve in $X$ meets $X_0$, because
	$X\setminus X_0$ is finite. Apply
	Lemma~\ref{snf-wp:elementary-properties}(3).
\end{proof}

The local freeness assumption cannot be replaced by reflexivity
(see \cite[Example~4.14]{La-Simpson}). The next example shows that the
surface assumption is also necessary, even for line bundles on
smooth varieties.

\begin{Example}\label{Fulger}
	Let $\cE:=\cO_{\PP^1}^{\oplus 2}\oplus \cO_{\PP^1} (-1)$,  $X=\PP(\cE )$ and  $L=\cO_{\PP(\cE)} (1)$. Let $C$ be the image of the section $\PP^1\to \PP(\cE)$ corresponding to $\cE\to \cO_{\PP^1} (-1)$. Then $H^0(X, L)= H^0(\PP^1, \cE)$ is $2$-dimensional and 
	the  linear system $|L|$ has a basis consisting of two  surfaces $\PP (\cO_{\PP^1}\oplus \cO_{\PP^1} (-1))$ corresponding to different factors of $\cO_{\PP^1}$ in $\cE$. These surfaces intersect along $C$, so the evaluation map $H^0(X, L)\otimes \cO_X\to L$ is surjective over $U=X\backslash C$. Since the degree of $L$ on $C$ is negative, $L$ is weakly positive over $U$ but it is not nef on $X$.
\end{Example}

\section{Operations on weakly positive vector bundles}
\label{sec:operations-weak-positivity}

Throughout this section, $R$ is a noetherian ring, $X$ is a
quasi-projective $R$-scheme, and $X_0\subseteq X$ is a fixed Zariski
dense open subset. In particular, mixed characteristic is allowed.
We fix an ample line bundle $H$ on $X$.

We will apply the theorems on tensor products and polynomial
representations of ample vector bundles from \cite{La-Ample}.
The following lemma explains how to pass from ordinary ampleness
to ampleness with respect to a fixed open subset.

\begin{Lemma}\label{operations:transfer-ampleness}
	Let $\Phi$ be an operation on finite tuples of vector bundles on
	quasi-projective $R$-schemes, taking values in vector bundles and
	commuting with arbitrary base change. Suppose that $\Phi$ takes
	tuples of ample vector bundles to ample vector bundles.
	If $\cE_1,\ldots,\cE_l$ are ample with respect to $X_0$, then
	$\Phi(\cE_1,\ldots,\cE_l)$ is ample with respect to $X_0$.
\end{Lemma}

\begin{proof}
	Choose homomorphisms
	\[
	u_i:\cO_X^{\oplus N_i}\longrightarrow
	\cW_i:=\Sym^{d_i}\cE_i\otimes H^{-1},
	\qquad 1\le i\le l,
	\]
	surjective over $X_0$. Work on the finitely many open-and-closed
	subsets on which all ranks are constant. Rank-zero factors can be
	omitted from the following Hom construction; if all ranks are zero,
	take $S=T=X$. For the remaining factors, put $r_i=\rk\cW_i>0$.
	By adding zero sections, arrange that $N_i\ge2r_i$.
	Let
	\[
	\pi:T=T_1\times_X\cdots\times_X T_l\longrightarrow X,
	\qquad
	T_i=\underline{\operatorname{Hom}}_X
	(\cO_X^{\oplus N_i},\cW_i).
	\]
	Let $j:S\hookrightarrow T$ be the open subset where all universal
	homomorphisms are surjective. The homomorphisms $u_i$ give a section
	$s:X\to T$ with $s(X_0)\subseteq S$.
	
	We have
	\begin{equation}\label{operations:hartogs}
		j_*\cO_S=\cO_T.
	\end{equation}
	This is immediate if $S=T=X$. Otherwise, work over an affine open subset trivializing the
	$\cW_i$. The coordinate ring $B$ of $T$ is a polynomial ring over
	the coordinate ring of that open subset. In the $i$-th universal
	matrix, choose maximal minors $\delta_i,\epsilon_i$ using disjoint
	blocks of $r_i$ columns. The polynomials
	\[
	f=\prod_i\delta_i,\qquad g=\prod_i\epsilon_i
	\]
	involve disjoint sets of variables and have leading coefficients
	equal to $1$, up to sign. Hence $f$ is a nonzerodivisor in $B$ and
	$g$ is a nonzerodivisor in $B/(f)$. Moreover,
	$D(f)\cup D(g)\subseteq S$.
	The extension property for a regular sequence of length two gives
	\[
	B\xrightarrow{\sim}
	\Gamma(D(f)\cup D(g),\cO).
	\]
	Every function on $S$ therefore extends to $T$, and the extension
	agrees with it on all of $S$, since $D(f)$ is schematically dense.
	The same argument after localization proves
	\eqref{operations:hartogs}. No reducedness assumption is used.
	
	Since $S\to X$ is quasi-affine, $\pi^*H|_S$ is ample
	\cite[Tag 0892]{SP}. The universal surjections show that
	\[
	(\pi^*H|_S)^{\oplus N_i}\twoheadrightarrow
	\Sym^{d_i}(\pi^*\cE_i|_S).
	\]
	Thus these symmetric powers, and hence $\pi^*\cE_i|_S$, are ample.
	The latter implication follows from the Veronese embedding.
	By the hypothesis on $\Phi$,
	\[
	\pi^*\Phi(\cE_1,\ldots,\cE_l)|_S
	\simeq\Phi(\pi^*\cE_1|_S,\ldots,\pi^*\cE_l|_S)
	\]
	is ample. Consequently, for some $m>0$,
	\[
	\pi^*\bigl(
	\Sym^m\Phi(\cE_1,\ldots,\cE_l)\otimes H^{-1}
	\bigr)|_S
	\]
	is globally generated.
	By \eqref{operations:hartogs}, sections of this locally free sheaf
	extend to $T$. Pulling the extensions back along $s$ proves the
	required generation over $X_0$.
\end{proof}

To treat weak positivity in mixed characteristic, we replace
Frobenius by finite flat covers extracting roots of a very ample
line bundle. The following uniform bound permits descent of the
required generation.

\begin{Lemma}\label{operations:finite-flat-generation}
	Let $f:Y\to X$ be finite and faithfully flat. There exists an
	integer $c\ge0$, depending only on $f$ and $H$, such that for
	every vector bundle $\cE$ on $X$ the following implication holds:
	if $f^*\cE$ is generated over $f^{-1}(X_0)$ by sections on $Y$,
	then $\cE\otimes H^c$ is generated over $X_0$ by sections on $X$.
\end{Lemma}

\begin{proof}
	The sheaf $\cB=f_*\cO_Y$ is locally free of positive rank.
	Choose $c\ge0$ such that
	$\cB\otimes\cB^\vee\otimes H^c$ is globally generated.
	Finitely many sections give a homomorphism
	$\cO_Y^{\oplus M}\to f^*\cE$ surjective over $f^{-1}(X_0)$.
	Pushing it forward gives
	\[
	\cB^{\oplus M}\longrightarrow\cE\otimes\cB,
	\]
	surjective over $X_0$.
	Tensoring with $\cB^\vee$ and composing with contraction gives
	\[
	(\cB\otimes\cB^\vee)^{\oplus M}\longrightarrow\cE,
	\]
	again surjective over $X_0$.
	Twisting by $H^c$ proves the assertion.
	The contraction $\cB\otimes\cB^\vee\to\cO_X$ is surjective
	because $\cB$ has positive rank; no invertibility of its rank is
	required.
\end{proof}

For the next proposition, choose an immersion
$i:X\hookrightarrow\PP_R^N$ and put $H=i^*\cO_{\PP_R^N}(1)$.

\begin{Proposition}\label{operations:root-cover-criterion}
	For a vector bundle $\cE$ on $X$, the following conditions are
	equivalent:
	\begin{enumerate}
		\item $\cE$ is weakly positive over $X_0$.
		\item For every integer $m>0$, there exist a finite faithfully flat
		morphism $f:Y\to X$ and a globally generated ample line bundle
		$\cA$ on $Y$ such that $f^*H\simeq\cA^m$
		{and} $f^*\cE\otimes\cA$	 is ample with respect to $f^{-1}(X_0)$.
		\item The condition in $(2)$ holds for arbitrarily large $m$.
	\end{enumerate}
\end{Proposition}

\begin{proof}
	For every $m>0$, the coordinate power morphism
	\[
	q_m:\PP_R^N\longrightarrow\PP_R^N,
	\qquad [z_0:\cdots:z_N]\longmapsto[z_0^m:\cdots:z_N^m],
	\]
	is finite and faithfully flat. On each standard affine chart its
	coordinate algebra is free over the target algebra, with basis
	the monomials whose exponents are all less than $m$.
	Base-changing $q_m$ along $i$ gives a finite faithfully flat
	morphism $f:Y\to X$ and an immersion $h:Y\hookrightarrow\PP_R^N$.
	Then
	\[
	\cA=h^*\cO_{\PP_R^N}(1)
	\]
	is ample and globally generated, and $f^*H\simeq\cA^m$ (this is a generalization of the standard Bloch-Gieseker covering
	\cite[Theorem 4.1.10]{Laz1}).
	
	Assume $(1)$. Choose $b>0$ such that
	$\Sym^{(m+1)b}\cE\otimes H^b$ is generated over $X_0$.
	Then
	\[
	\Sym^{(m+1)b}(f^*\cE\otimes\cA)\otimes\cA^{-1}
	\simeq
	f^*(\Sym^{(m+1)b}\cE\otimes H^b)\otimes\cA^{b-1}
	\]
	is generated over $f^{-1}(X_0)$. This proves $(2)$.
	In fact, this argument works for every cover and line bundle
	satisfying the stated root and global-generation conditions.
	The implication $(2)\Rightarrow(3)$ is immediate.
	
	Assume $(3)$ and fix $a>0$. Choose $m>a$ and corresponding
	$f,\cA$, and let $c$ be the integer from
	Lemma~\ref{operations:finite-flat-generation}. Put
	$\cV=f^*\cE\otimes\cA$.
	For $b\gg0$, the identity
	\[
	f^*(\Sym^{ab}\cE\otimes H^{b-c})
	\simeq
	\Sym^{ab}\cV
	\otimes\cA^{(m-a)b-mc}
	\]
	shows that its left-hand side is generated over $f^{-1}(X_0)$:
	the first factor on the right is generated there by
	Lemma~\ref{open-ample-equivalence}(3), and the second is globally
	generated once its exponent is nonnegative.
	Lemma~\ref{operations:finite-flat-generation} now shows that
	$\Sym^{ab}\cE\otimes H^b$ is generated over $X_0$.
	Thus $(1)$ holds.
\end{proof}

\begin{Corollary}\label{operations:weak-positivity}
	Each of the following two classes of vector bundles on $X$ is
	closed under finite direct sums, quotient vector bundles, tensor
	products, positive symmetric powers, positive divided powers, and
	exterior powers of positive rank:
	\begin{enumerate}
		\item vector bundles ample with respect to $X_0$;
		\item vector bundles weakly positive over $X_0$.
	\end{enumerate}
	More generally, let $\cE$ have constant rank $r$, and let $W$ be a
	finite locally free polynomial $\GL_{r,R}$-module with vanishing
	degree-zero homogeneous component. If $\cE$ belongs to either
	class, then its associated bundle $\cE(W)$ belongs to that class
	whenever it has positive rank.
\end{Corollary}

\begin{proof}
	For ordinary ample vector bundles, direct sums and quotient
	bundles preserve ampleness. The assertions about tensor products
	and polynomial representations are the ampleness theorems of
	\cite{La-Ample}. All the displayed operations commute with base
	change. Thus Lemma~\ref{operations:transfer-ampleness} proves the
	assertions about ampleness with respect to $X_0$; closure under
	quotient bundles follows directly from the definition.
	
	For weak positivity, it suffices first to treat an operation
	$\Phi$ on $l$ bundles which is homogeneous of positive total
	degree $d$ under simultaneous line-bundle twists, so that
	\[
	\Phi(\cE_1\otimes\cL,\ldots,\cE_l\otimes\cL)
	\simeq\Phi(\cE_1,\ldots,\cE_l)\otimes\cL^d.
	\]
	Suppose that the $\cE_i$ are weakly positive over $X_0$.
	Given $m>0$, take the coordinate root cover of order $md$ from the
	proof of Proposition~\ref{operations:root-cover-criterion}.
	It gives $f:Y\to X$ and an ample globally generated $\cA$ with
	$f^*H=\cA^{md}$. The necessity argument in that proposition
	shows simultaneously that every $f^*\cE_i\otimes\cA$ is ample
	with respect to $f^{-1}(X_0)$. Hence
	\[
	\Phi(f^*\cE_1\otimes\cA,\ldots,f^*\cE_l\otimes\cA)
	\simeq
	f^*\Phi(\cE_1,\ldots,\cE_l)\otimes\cA^d
	\]
	is ample with respect to $f^{-1}(X_0)$.
	Since $(\cA^d)^m=f^*H$, the sufficiency direction of
	Proposition~\ref{operations:root-cover-criterion} proves that
	$\Phi(\cE_1,\ldots,\cE_l)$ is weakly positive over $X_0$.
	
	Apply this argument to direct sums ($d=1$), tensor products
	($d=2$), and $\Sym^d$, $\Gamma^d$, and $\bigwedge^d$. 
	Quotient bundles preserve weak positivity because symmetric
	powers preserve surjections.
	Finally, decompose $W$ into its finitely many homogeneous
	components, apply the argument to each positive degree, and use
	closure under direct sums.
\end{proof}

\section{Positivity in positive characteristic}

In this section $X$ is a quasi-projective scheme defined over  an algebraically closed field $k$
of positive characteristic $p$.

\subsection{Frobenius criteria for weak positivity}
\label{sec:frobenius-wp}

We have the following positive characteristic analogue of \cite[Lemma 2.27]{Vi2}.

\begin{Proposition}\label{frobenius-wp} 
	Let $\cE$ be a vector bundle of positive rank on $X$ 
	and $H$ be an ample line bundle on $X$.	  Fix a Zariski dense open subset $X_0\subseteq X$.
	The following conditions are equivalent:
	\begin{enumerate}
		\item $\cE$ is weakly positive over $X_0$.
		\item For every $e\ge0$, the bundle $(F_X^e)^*\cE\otimes H$ is ample
		with respect to $X_0$.
		\item For arbitrarily large $e$, the bundle $(F_X^e)^*\cE\otimes H$ is
		ample with respect to $X_0$.
	\end{enumerate}
\end{Proposition}

\begin{proof}
	Assume $(1)$, fix $e\ge0$, and put $q=p^e$.
	Choose $b>0$ such that $\Sym^{2qb}\cE\otimes H^b$ is generated over
	$X_0$. Replacing $b$ by a sufficiently large multiple, assume that
	$H^{qb-1}$ is globally generated.
	Then the sheaf
	\[
	\Sym^{2qb}((F_X^e)^*\cE\otimes H)\otimes H^{-1}
	=
	(F_X^e)^*(\Sym^{2qb}\cE\otimes H^b)\otimes H^{qb-1}
	\]
	is generated over $X_0$. This proves $(2)$, and $(2)\Rightarrow(3)$
	is immediate.
	
	Assume $(3)$ and fix $a>0$. Choose $e$ with $q=p^e>a$
	such that \(	\cG=(F_X^e)^*\cE\otimes H\)
	is ample with respect to $X_0$, and let $r$ be the maximum of the ranks of $\cE$ on the connected components of $X$.
	The natural map $(F_X^e)^*\cE\to\Sym^q\cE$, which sends a local
	section to its $q$th power, induces for $m\ge r-1$ a surjection
	\begin{equation}\label{frob-monomials}
		\bigoplus_{j=0}^{r-1} 
		\Sym^{m-j}((F_X^e)^*\cE)\otimes\Sym^{qj}\cE
		\longrightarrow\Sym^{qm}\cE.
	\end{equation}
	Indeed, in a local basis write every exponent of a monomial as
	$qu_i+v_i$, with $0\le v_i<q$. Since the total degree is $qm$, the
	sum of the $v_i$ is $qj$ for some $0\le j<r$.
	
	Put $m=an$ and tensor
	\eqref{frob-monomials} by $H^{qn}$. Its source becomes
	\[
	\bigoplus_{j=0}^{r-1}
	\Sym^{an-j}\cG\otimes
	\bigl(\Sym^{qj}\cE\otimes H^{(q-a)n+j}\bigr).
	\]
	For $n\gg0$, the first factor of each summand is generated over $X_0$
	by Lemma~\ref{open-ample-equivalence}, while the second is globally
	generated on $X$, since $q>a$. Hence
	$\Sym^{aqn}\cE\otimes H^{qn}$ is generated over $X_0$.
	Taking $b=qn$ proves $(1)$.
\end{proof}

\begin{Remark}
	For $X_0=X$, ampleness with respect to $X_0$ is ordinary ampleness, so this includes the criterion for weak positivity over all of $X$.
\end{Remark}

\begin{Example}\label{restriction-example} In
	Proposition~\ref{frobenius-wp}, ampleness with respect to $X_0$ cannot be replaced  by ampleness of the
	restriction to $X_0$. Take
	\[
	X=\PP_k^1,\qquad X_0=\AA_k^1,\qquad
	\cE=\cO_{\PP^1}(-1),\qquad H=\cO_{\PP^1}(1).
	\]
	Every $((F_X^e)^*\cE\otimes H)|_{X_0}$ is trivial, hence ample
	on the affine scheme $X_0$. But, for every $b>0$,
	\[
	\Sym^{2b}\cE\otimes H^b=\cO_{\PP^1}(-b)
	\]
	has no global sections. Thus $\cE$ is not weakly positive over $X_0$.
\end{Example}

\subsection{Weakly $p$-positive sheaves}

Let us recall the following notion of ampleness from \cite{Gi} and \cite{Ha1}.

\begin{Definition}
	A vector bundle $\cE$ on $X$ is called \emph{$p$-ample} if for any coherent $\cO_X$-module $\cF$, $(F_X^{e})^{*}\cE\otimes \cF$ is globally generated for $e\gg 0$.
\end{Definition}

Similarly, motivated by Viehweg's definition of weakly positive sheaves  we consider the following notions.

\begin{Definition}
	Assume that $X$ is normal.
	Let $\cE$ be a coherent torsion free $\cO_X$-module and let $V_{\cE}$ be the largest open subset of $X$
	on which $\cE$ is locally free.
	\begin{enumerate}
		\item 
		We say that $\cE$ is \emph{weakly $p$-positive over   a Zariski dense open subset $X_0\subset V_{\cE}$} if for every ample line bundle $H$ on $X$ and 
		every positive integer $a$ there exists  a positive integer $b $ such that $F_X^{[ab]}\cE \otimes H^{b}$ is globally generated over $X_0$.
		\item 	We say that $\cE$ is \emph{weakly $p$-positive} if there exists  a Zariski dense open subset $X_0\subset V_{\cE}$ such that $\cE$ is weakly $p$-positive over $X_0$.
	\end{enumerate}
\end{Definition}

The following proposition is motivated by \cite[Proposition 2.9]{Vi2}.

\begin{Proposition}\label{projective-p-positivity}
	Let $\cE$ be a vector bundle on a projective $k$-scheme $X$. Consider the following conditions:
	\begin{enumerate}
		\item $\cE$ is nef.
		\item $\cO_{\PP(\cE)} (1)$ is nef.
		\item For one ample line bundle $H$ on $X$ and for every $e>0$,
		\( (F_X^{e})^*\cE\otimes H \)
		is $p$-ample.
		\item For all ample line bundles $H$ on $X$ and for every $e>0$,
		\(
		(F_X^{e})^*\cE\otimes H
		\)
		is $p$-ample.
		\item For one ample line bundle $H$ on $X$ and for every positive integer $a$ there exists  a positive integer $b$ such that $(F_X^{ab})^*\cE \otimes H^{ b}$ is globally generated.
	\end{enumerate}
	Then
	\[
	(5)\Longrightarrow(4)\Longleftrightarrow(3)
	\Longrightarrow(1)\Longleftrightarrow(2).
	\]
\end{Proposition}

\begin{proof}
	Equivalence of $(1)$ and $(2)$ is the definition of nefness, and
	$(4)\Rightarrow(3)$ is immediate. We use the elementary facts that
	$p$-ampleness is preserved by quotients and by tensoring with a globally
	generated vector bundle, that an ample line bundle is $p$-ample, and that
	a vector bundle is $p$-ample if and only if any one of its Frobenius
	pullbacks is $p$-ample.
	
	Assume $(3)$, and let $\cA$ be any ample line bundle. Fix $e>0$ and
	choose $n$ so large that $\cA^{p^n}\otimes H^{-1}$ is globally
	generated. Then
	\[
	(F_X^n)^*((F_X^e)^*\cE\otimes\cA)
	=
	((F_X^{e+n})^*\cE\otimes H)
	\otimes(\cA^{p^n}\otimes H^{-1})
	\]
	is $p$-ample. This proves $(4)$.
	
	To prove $(5)\Rightarrow(4)$, fix an ample line bundle $\cA$ and $e>0$.
	Choose $c>0$ such that $\cA^c\otimes H^{-1}$ is ample, and then choose
	$a>e$ such that $p^{a-e}>c$. For every $b\ge1$ we have
	$p^{ab-e}>cb$. Choose $b$ as in $(5)$. The line bundle
	\[
	\cA^{p^{ab-e}}\otimes H^{-b}
	=
	(\cA^c\otimes H^{-1})^b\otimes\cA^{p^{ab-e}-cb}
	\]
	is ample. Thus
	\[
	(F_X^{ab-e})^*((F_X^e)^*\cE\otimes\cA)
	=
	((F_X^{ab})^*\cE\otimes H^b)
	\otimes(\cA^{p^{ab-e}}\otimes H^{-b})
	\]
	is a quotient of a finite direct sum of an ample line bundle, and hence
	is $p$-ample. This proves $(4)$.
	
	Finally, a $p$-ample vector bundle is ample by \cite[Proposition 6.3]{Ha1}.
	Under $(3)$, if $f:C\to X$ is a morphism from a smooth projective curve
	and $\cL$ is a line bundle quotient of $f^*\cE$, then
	\[
	p^e\deg\cL+\deg f^*H\ge0
	\]
	for every $e>0$. Dividing by $p^e$ and letting $e$ tend to infinity proves
	that $\cE$ is nef.
\end{proof}

\begin{Remark*}
	The implication $(1)\Rightarrow(4)$ is false in general. Let $\cG$ be
	Gieseker's ample vector bundle on $\PP^2$ which is not $p$-ample
	\cite[Theorem 3.1]{Gi}, and fix an ample line bundle $H$.
	By Lemma~\ref{ample-Frobenius-twist}, for some $n$ the bundle
	$\cE=(F_X^n)^*\cG\otimes H^{-1}$ is ample, hence nef. But
	\[
	(F_X)^*\cE\otimes H^p=(F_X^{n+1})^*\cG
	\]
	is not $p$-ample. Thus the five conditions above are not equivalent.
\end{Remark*}

\subsection{Ample and $p$-ample vector bundles}

\begin{Lemma}\label{p-ample-equivalence}
	Let $\cE$ be a vector bundle on $X$. The following conditions are equivalent. 
	\begin{enumerate}
		\item $\cE$ is $p$-ample.
		\item For some ample line bundle $H$ on $X$ the bundle $(F_X^{e})^{*}\cE\otimes H^{-1}$ is globally generated for $e\gg 0$.
	\end{enumerate}
\end{Lemma}

\begin{proof}
	By Serre's theorem for any coherent $\cO_X$-module $\cF$ there exists $m_0$ such that $\cF \otimes H^{ m}$ is globally generated
	for every $m\ge m_0$. 
	By assumption there exists some $e_0$ such that for every $e\ge e_0$ the bundle $(F_X^{e})^{*}\cE\otimes H^{-1}$ is globally generated.
	It follows that $$(F_X^{e+n})^{*}\cE\otimes  H^{-p^n}=(F_X^n)^*((F_X^{e})^{*}\cE\otimes  H^{-1})$$ is also globally generated for all $e\ge e_0$ and $n\ge 0$. So if $p^n\ge m_0$ then
	$$(F_X^{e+n})^{*}\cE\otimes \cF= ((F_X^{e+n})^{*}\cE\otimes _{\cO_X} H^{-p^n})\otimes (\cF \otimes H^{p^n}) $$ 
	is globally generated for every $e\ge e_0$. 
\end{proof}

The following proposition was proven for projective schemes in \cite[Corollary 6.7]{Ha1}.

\begin{Proposition}
	Let $\cE_1$ and $\cE_2$ be $p$-ample vector bundles on $X$. Then $\cE_1\otimes  \cE_2$ is also $p$-ample.
\end{Proposition}

\begin{proof}
	Let $H$ be an ample line bundle on $X$. By assumption there exists some $e_0$ such that for 
	every $e\ge e_0$ the bundles  $(F_X^{e})^{*}\cE_1\otimes  H^{-1}$ and
	$(F_X^{e})^{*}\cE_2\otimes H^{-1}$ are globally generated. Therefore  
	$(F_X^{e})^{*}(\cE_1\otimes \cE_2)\otimes  H^{-2}$ is globally generated  for
	all $e\ge e_0$. So 
	by Lemma \ref{p-ample-equivalence} the bundle
	$\cE_1\otimes  \cE_2$ is $p$-ample.
\end{proof}

\medskip

\medskip

\begin{Lemma}\label{ample-Frobenius-twist}
	Let $\cE$ be an ample vector bundle on $X$ and let $H$ be an ample line bundle on $X$. Then  there exists some $e_0$ such that  $(F_X^{e})^{*}\cE\otimes H^{-1}$ is ample for all  $e\ge e_0$.
\end{Lemma}

\begin{proof}
	By Lemma \ref{Viehweg-ampleness} there exists some $d>0$ such that $\Sym ^d\cE\otimes  H^{-1}$ is globally generated over $X$.  Then 
	\begin{equation*}
		\Sym ^d((F_X^e)^*\cE\otimes  H^{-1})\otimes  H^{{d}-p^e}
		=(F_X^e)^*(\Sym ^d\cE\otimes H^{-1})
	\end{equation*}
	is also globally generated. So if $p^e>{d}$ then $(F_X^{e})^{*}\cE\otimes  H^{-1}$ is ample by Lemma \ref{Viehweg-ampleness}.
\end{proof}

As a corollary we obtain a generalization of Hironaka's result \cite[Proposition 3.1]{Ba} to quasi-projective schemes. Note that the proof of \cite[Proposition 3.1]{Ba} uses a numerical characterization of ampleness so it does not extend to quasi-projective schemes.

\begin{Corollary}
	Let $\cE$ and $\cF$ be  vector bundles on $X$.
	If $\cE$ is ample then		
	there exists some $e_0$ such that  $(F_X^{e})^{*}\cE\otimes  \cF$ is ample for all  $e\ge e_0$.
\end{Corollary}

\begin{proof}
	Fix an ample line bundle $H$ on $X$, and choose $m>0$ such that $\cF\otimes H^m$ is globally generated.
	By Lemma~\ref{ample-Frobenius-twist}, applied to $H^m$,
	the bundle $(F_X^e)^*\cE\otimes H^{-m}$ is ample for $e\gg0$.
	Hence $(F_X^e)^*\cE\otimes\cF$ is a quotient of a finite direct sum
	of this ample bundle.
\end{proof}

\begin{Remark}
	In	\cite[Theorem 3.1]{Gi} Gieseker constructed an ample vector bundle on $\PP^2$, which is not $p$-ample. In fact, in his example $(F_X^{e})^{*}\cE$ is not $p$-ample for any $e\ge 0$.
	For this bundle $\cE$ and every ample line bundle $H$, there is no $e\ge0$ such that $(F_X^{e})^{*}\cE\otimes H^{-1}$ is $p$-ample.
\end{Remark}

\section{Strong numerical flatness and weak positivity}
\label{sec:snf-weak-positivity}

In this section $X$ is a normal quasi-projective variety defined over an
algebraically closed field.

\subsection{S-fundamental group scheme for normally big varieties}

\begin{Definition}\label{def:weakly-flat} 
	We say that a vector bundle $\cE$ on a variety $X$ is
	\emph{weakly flat} if both $\cE$ and $\cE^*$ are weakly positive
	over $X$.
	Let $U\subset X$ be a dense open subset.
	We say that a coherent sheaf $\cF$ on $X$ is
	\emph{weakly flat over $U$} if $\cF|_U$ is locally free and
	both $\cF$ and $\cF^*$ are weakly positive over $U$.
\end{Definition}

We denote by $\Vect^{\wf}(X)$ the full subcategory of coherent
$\cO_X$-modules whose objects are weakly flat vector bundles.
We denote by $\Vect^{\wf}(X/U)$ the full subcategory whose objects
are reflexive coherent sheaves weakly flat over $U$.
In particular, we have
$\Vect^{\wf}(X)=\Vect^{\wf}(X/X)$.
Moreover, if $j:X\hookrightarrow\overline X$ is a small
compactification, then restriction and
reflexive extension give mutually inverse equivalences
\[
\Vect^{\wf}(\overline X/X)\simeq\Vect^{\wf}(X),
\qquad
\cF\longmapsto j^*\cF,
\qquad
\cE\longmapsto j_*\cE.
\]

\begin{Lemma}\label{snf-wp:weak-flat-determinant}
Let $\cE$ be a vector bundle on $X$. Then the following conditions are equivalent:
	\begin{enumerate}
		\item $\cE$ is weakly flat.
		\item Both $\cE$ and $\det\cE^*$ are weakly positive over $X$.
	\end{enumerate}
\end{Lemma}

\begin{proof}
If $\cE$ is weakly flat, then $\det \cE^*$ is weakly positive
	by Corollary \ref{operations:weak-positivity}. This proves
	$(1)\Rightarrow(2)$.
	Conversely, assume (2) and let $r:=\rk \cE$. Then Corollary \ref{operations:weak-positivity}
	implies that both  $\bigwedge^{r-1}\cE$ and
	\[
	\cE^*\simeq
	{\bigwedge}^{r-1}\cE\otimes\det\cE^*
	\]
are weakly positive. Thus $\cE$ is weakly flat.
\end{proof}

From now on in this subsection we assume that $X$ admits a normal projective compactification  $j:X\hookrightarrow \bar X$ such that the complement of $X$ in $\bar X$ has codimension $\ge 2$.

\begin{Proposition}\label{Vect-wf-rigid-abelian}
	$\Vect ^{\wf}(X)$ is a rigid  abelian $k$-linear tensor category.
\end{Proposition}
\begin{proof}
	By Corollary~\ref{operations:weak-positivity}, the category
	$\Vect^{\wf}(X)$ is closed under finite direct sums and tensor
	products. Let $\cE_1$ and $\cE_2$ be weakly flat vector bundles
	on $X$ and let $\varphi:\cE_1\to\cE_2$ be an
	$\cO_X$-linear morphism. Note that a weakly flat vector bundle is numerically flat in the sense of \cite[Definition 3.1]{La-Simpson}. So  \cite[Proposition 3.4]{La-Simpson} shows that 
	\[
	\cK:=\ker\varphi,\qquad
	\cI:=\operatorname{im}\varphi,\qquad
	\cQ:=\operatorname{coker}\varphi
	\]
	are locally free.
	
	The bundles $\cI$ and $\cI^*$ are quotients of $\cE_1$ and
	$\cE_2^*$, respectively. Hence both are weakly positive, and
	$\cI$ is weakly flat. Similarly, $\cK^*$ and $\cQ$ are weakly
	positive as quotients of $\cE_1^*$ and $\cE_2$, respectively.
	The determinant identities
	\[
	\det\cK\simeq\det\cE_1\otimes\det\cI^*,
	\qquad
	\det\cQ^*\simeq\det\cE_2^*\otimes\det\cI
	\]
	show that $\det\cK$ and $\det\cQ^*$ are weakly positive.
	Applying Lemma~\ref{snf-wp:weak-flat-determinant} to
	$\cK^*$ and $\cQ$, we conclude that $\cK$ and $\cQ$ are
	weakly flat.
	
	Thus $\Vect^{\wf}(X)$ is closed under kernels and cokernels
	in the category of coherent $\cO_X$-modules, and hence is
	abelian. It is $k$-linear, contains $\cO_X$, and is closed
	under duals by definition. The usual evaluation and
	coevaluation morphisms therefore make it a rigid tensor
	category.
\end{proof}

\medskip

Every $k$-point $x$ of $X$ gives a fiber functor $T_x:\Vect ^{\wf}(X) \to k\operatorname{-}\Vect$ defined by sending $\cE$ to $\cE\otimes k(x)$. Then  $(\Vect ^{\wf}(X), \otimes, T_x, \cO_X)$ is a neutral Tannaka category
by Proposition \ref{Vect-wf-rigid-abelian}.

\begin{Definition}\label{def:S-fundamental}
	The affine $k$-group scheme Tannaka dual to  $(\Vect ^{\wf}(X), \otimes, T_x, \cO_X)$ is called the \emph{S-funda\-men\-tal group scheme} of $X$ and it is denoted by $\pi_1^S(X, x)$.
\end{Definition}

In positive characteristic, Theorem \ref{snf-wp:equivalence} shows that the above definition is equivalent  to \cite[Definition 4.18]{La-Simpson}.

\subsection{Comparison of S-fundamental group schemes of open subsets}

The following proposition generalizes  a part of \cite[Proposition 4.20]{La-Simpson} to an arbitrary characteristic.

\begin{Proposition} \label{prop:comparison-of-fund-groups}
Let $X$ be a normally big $k$-variety with a small compactification  $j:X\hookrightarrow \bar X$.
Let $X'\subseteq\bar X$ be an open subset containing $X$,
and let $x\in X(k)$. Then we have an induced faithfully flat homomorphism
	$$\pi_1^S(X, x)\to \pi_1^S(X', x)$$
	of S-fundamental group schemes.
\end{Proposition}

\begin{proof} Let $i: 	X\hookrightarrow X'$ 	be the open embedding factoring $j$. 
	Restriction defines an exact tensor functor
	\[
	i^*:\Vect^{\wf}(X')\longrightarrow\Vect^{\wf}(X)
	\]
	compatible with evaluation at $x$. Since $X'$ is normal and
	$X'\setminus X$ has codimension at least two, this functor is
	fully faithful. 	So by \cite[Proposition~2.21(a)]{DM}
it is sufficient to prove that the essential image of $i^*$ is closed
under subobjects.
Let $\cE'\in\Vect^{\wf}(X')$ and let
	\[
	0\longrightarrow\cG\longrightarrow i^*\cE'
	\longrightarrow\cQ\longrightarrow0
	\]
	be an exact sequence in $\Vect^{\wf}(X)$. We may assume that
	$q:=\rk\cQ>0$. The quotient defines a section
	\(
	X\longrightarrow\operatorname{Gr}(\cE', q)\) of the Grassmannian of rank $q$ quotient vector bundles of $\cE'$
	(see \cite[Example 2.2.3]{Huybrechts-Lehn2010}).
	Let $\Gamma$ be the normalization of its closure, with
	structure morphism $p:\Gamma\to X'$. Thus $p$ is projective
	and birational and is an isomorphism over $X$. Denote the
	tautological quotient on $\Gamma$ by $\cQ_\Gamma$.
	
	Choose a projective compactification of $\Gamma$ over $\bar X$
	and a smooth projective alteration dominating it. Write the
	Stein factorization of the resulting morphism as
	\[ 
	h:Y\xrightarrow{f}Z\xrightarrow{g}\bar X,
	\]
	where $f$ is birational and $g$ is finite. Set $Y'=h^{-1}(X')$ and $Z'=g^{-1}(X').$
	Let $\cL$ be a line bundle on $Y$ extending the pullback of
	$\det\cQ_\Gamma$ to $Y'$.

Choose a very ample line bundle \(H\) on $\overline{X}$.	
	The determinant of $\cQ$ and its inverse are weakly positive.
	Their pullbacks to $g^{-1}(X)$ therefore have degree zero on
	general complete intersection curves. Applying Lemma \ref{lem:numerical-criterion} to their reflexive extension on $Z$, with
	polarization $g^*H$, gives
	\(
	c_1(\cL)\equiv D
	\)
	for an $f$-exceptional $\mathbb Q$-divisor $D$.
	
	On $Y'$, the line bundle $\cL$ is a quotient of
	$\bigwedge^q h^*\cE'$. Its degree on every complete curve in
	$Y'$ is therefore nonnegative: after normalization, the
	pullback of $\cE'$ and its dual are nef on that projective
	curve. In particular, $D|_{Y'}$ is $f$-nef. The negativity
	lemma gives
	\(
	D|_{Y'}\le0.
	\)
	
	In fact, $D|_{Y'}=0$. To see this, note that $Z'$ is connected
	by proper curves, since it is a big open subset of the normal
	projective variety $Z$. As $f$ is proper with connected
	fibers, $Y'$ is connected by proper chains of curves.
	If $-D|_{Y'}$ were nonzero, such a chain joining a point outside
	its support to a point in its support would contain an
	integral complete curve $C$ meeting that support without
	being contained in it. Then
	\(
	\deg_C\cL=D\cdot C<0,
	\)
	a contradiction.
	
	The morphism $p$ has no positive-dimensional fiber. Otherwise
	a curve in such a fiber could be dominated by a complete
	curve in $Y'$. The degree of $\cL$ on this curve would be
	positive, by relative ampleness of the Pl\"ucker line bundle,
	whereas $D|_{Y'}=0$ makes it zero. Thus $p$ is finite and
	birational, hence an isomorphism by normality of $X'$.
	We obtain an exact sequence of vector bundles
	\[
	0\longrightarrow\cG'\longrightarrow\cE'
	\longrightarrow\cQ'\longrightarrow0
	\]
	extending the given sequence on $X$.
	
	We must still verify weak positivity at the added points.
	Put $\cN'=\det\cQ'$. Choose $d>0$ clearing the denominators
	of $D$, and set
	\(
	\cM:=\cL^{\otimes d}\otimes\cO_Y(-dD).
	\)
	This line bundle is numerically trivial, and
	\[
	\cM|_{Y'}\simeq h^*(\cN'^{\otimes d}).
	\]
	The line bundles $\cM^{\otimes m}$, $m\in\mathbb Z$, form a
	bounded family (see \cite[Theorem 9.6.3]{KleimanPicard}).
	Boundedness is preserved by proper pushforward and by taking
	reflexive determinants. Hence the rank-one reflexive sheaves
	\[
	\cT_m:=
	\left(
	\det h_*(\cM^{\otimes m})
	\otimes(\det h_*\cO_Y)^*
	\right)^{**},
	\qquad m\in\mathbb Z,
	\]
	form a bounded family on $\bar X$.
	
	Let $e=\deg g$. The projection formula gives
	\(
	\cT_m|_{X'}\simeq\cN'^{\otimes dem}.
	\)
 After replacing $H$ by a very ample power, choose $c>0$,
divisible by $de$, such that every $\cT_m\otimes H^c$ is
	globally generated. For any $a>0$, taking
	$m=\pm ac/(de)$ shows that
	\[
	\cN'^{\otimes ac}\otimes H^c|_{X'},
	\qquad
	\cN'^{\otimes(-ac)}\otimes H^c|_{X'}
	\]
	are globally generated. Thus $\cN'$ is weakly flat.
	
	Now $\cQ'$ and $\cG'^*$ are weakly positive, being quotients
	of $\cE'$ and $\cE'^*$, respectively. Moreover, $\det\cQ'^*=\cN'^*$ and
	$\det\cG'=\det\cE'\otimes\cN'^*$
	are weakly positive. So Lemma~\ref{snf-wp:weak-flat-determinant}
	implies that both $\cQ'$ and $\cG'$ are weakly flat.
	This proves closure of the essential image under subobjects.
\end{proof}

\begin{Remark}
We expect that if $X'$ is smooth then $\pi_1^S(X, x)\to \pi_1^S(X', x)$ in the above proposition is an isomorphism. If $k$ has positive characteristic this is true by \cite[Proposition 4.20]{La-Simpson}. Unfortunately, it is not clear how to prove this in characteristic zero.
\end{Remark}

\subsection{The equivalence in positive characteristic}

\begin{Theorem}\label{snf-wp:equivalence} 
	Let $X$ be a normal  variety of dimension $n\ge 2$ defined over an algebraically closed
	field of characteristic $p>0$. Assume that $X$ admits a small compactification  $j:X\hookrightarrow\overline X$
and let $H$ be an ample line bundle on $\overline X$.
Let $\cE$ be a vector bundle on $X$, and set $\cF:=j_*\cE$.  The following conditions are equivalent:
	\begin{enumerate}
		\item $\cE$ is strongly numerically flat.
		\item  $\cE$ is weakly flat.
		\item $\cF$ is strongly slope $H$-semistable and
		\[ 
		c_1(\cF)\cdot H^{n-1}=0,
		\qquad
		\int_{\overline X}\ch_2(\cF)H^{n-2}=0. 
		\]
	\end{enumerate}
\end{Theorem}

\begin{proof}
We may assume that $r:=\rk\cE>0$.
By \cite[Corollary~4.6]{La-Simpson}, condition (1) implies (3) (in fact, this implication follows easily from the definition). 
Moreover, under either (1) or (3), the family
\(\{F_{\overline X}^{[a]}\cF\}_{a\ge0}\)
is bounded (these implications are non-trivial). Since $\cE$ is locally free on the normal variety
$X$, we have
\[
F_{\overline X}^{[a]}\cF
\simeq j_*(F_X^a)^*\cE.
\]
Indeed, both sides are reflexive and agree on $X$.
Their duals are the sheaves $j_*(F_X^a)^*\cE^*$, and form a
bounded family by \cite[Proposition~2.5]{La-Simpson}.

Consequently, there is one integer $c\ge0$ such that both
\[
j_*(F_X^a)^*\cE\otimes H^c
\quad\hbox{and}\quad
j_*(F_X^a)^*\cE^*\otimes H^c
\]
are globally generated for every $a\ge0$.
After restriction to $X$ and tensoring with $H|_X$, the
resulting bundles are locally free quotients of finite direct
sums of the ample line bundle $H|_X$, and hence are ample.
Proposition \ref{frobenius-wp} applied with 
$H^{c+1}|_X$, proves weak positivity of both $\cE$ and $\cE^*$.
Thus either (1) or (3) implies (2).	
	
Assume (2). Lemma~\ref{snf-wp:elementary-properties} shows that 
	$\cE$ is numerically flat. Let $g:V\to X$ be a morphism from a
	normally big surface, and let $i:V\hookrightarrow S$ be its
	small compactification. Put
	\[
	\cV=g^*\cE,\qquad \cG=i_*\cV,\qquad D=c_1(\cG).
	\]
	Both $\cV$ and $\cV^*$ are weakly positive over $V$, so
	$\cV$ is numerically flat by Lemma \ref{snf-wp:elementary-properties}. Choose an ample line bundle $H$ on
	$S$. 
	General smooth curves in sufficiently large multiples of $H$
	lie in $V\cap S_{\mathrm{reg}}$. Numerical flatness on these
	curves gives $D\cdot H=0$ and strong $H$-semistability of $\cG$.
	Indeed, for each $a$ and each saturated subsheaf of
	$F_S^{[a]}\cG$, choose such a curve avoiding the non-locally-free
	loci of that subsheaf and its quotient. The ambient bundle on
	the curve is a Frobenius pullback of a numerically flat bundle,
	hence is semistable of degree zero. Thus the subsheaf has
	non-positive slope. 
	
	The Hodge index theorem and Bogomolov's inequality on normal
	projective surfaces give
	\begin{equation}\label{snf-wp:hodge-bogomolov}
		D^2\le0,\qquad
		\Delta(\cG)=D^2-2r\int_S\ch_2(\cG)\ge0;
	\end{equation}
	(see \cite[Corollary~6.6]{La-Inters}).
	By Proposition \ref{frobenius-wp} the bundle 
	$(F_V^a)^*\cV\otimes H|_V$ is ample for every $a\ge0$.
	Its reflexive extension is
	\(
	\cG_a:=F_S^{[a]}\cG\otimes H.	\)
Then Lemma~\ref{snf-wp:ample-surface-segre} implies that
\begin{equation*}
		0<s_2(\cG_a) =q^2\left(\int_S\ch_2(\cG)+\frac12D^2\right)
	+\frac{r(r+1)}2H^2,
\end{equation*}
where $q=p^a$. 
	Dividing by $q^2$ and letting $a$ tend to infinity yields
	$\int_S\ch_2(\cG)\ge-\frac12D^2$. Together with
	\eqref{snf-wp:hodge-bogomolov}, this gives
	\[
	0\le-\frac12D^2\le\int_S\ch_2(\cG)
	\le\frac{D^2}{2r}\le0.
	\]
	Hence $D^2=0$ and $\int_S\ch_2(i_*g^*\cE)=0$.
	Since $g$ was arbitrary, $\cE$ is strongly numerically flat.
\end{proof}

\begin{Remark}
Note that the proof of Theorem~\ref{snf-wp:equivalence} is independent of Corollary~\ref{operations:weak-positivity}.
So under the hypotheses of the theorem, closure of weakly flat bundles under tensor products follows already from the
corresponding property of strongly numerically flat bundles \cite[Proposition~4.15]{La-Simpson}.
\end{Remark}

\subsection{Extension to projective compactifications}

The results in this subsection hold in arbitrary characteristic,
unless otherwise indicated. The following technical lemma is needed in the proof of Theorem \ref{nf-normal:prop-weak-positivity}.

\begin{Lemma}\label{lem:nf-birational-chern}
	Let $\pi:T\to S$ be a projective birational morphism from a
	smooth projective surface to a normal projective surface.
	Let $\cE$ be a numerically flat vector bundle on $T$ and set
	$\cG:=(\pi_*\cE)^{**}$. Then $c_1(\cG)\equiv 0$ and $\int_S\ch_2(\cG)=0.$
\end{Lemma}

\begin{proof}
	Since $\cE$ is numerically flat, its Chern classes are
	numerically trivial. In particular, its first relative
	Chern class vanishes at every exceptional fiber of $\pi$.
	The formula for the first Chern class therefore gives
	$c_1(\cG)\equiv0$.
	By the definition of the second Chern class on a normal
	surface, it remains to prove
	\[
	c_2(\pi_x,\cE)=0
	\qquad\text{for every exceptional fiber over }x\in S.
	\]
	Relative Bogomolov's inequality,
	\cite[Proposition~3.2(2)]{La-Inters}, gives
	$c_2(\pi_x,\cE)\geq0$.
	The assertion in rank one follows from
	\cite[Proposition~3.2(1)]{La-Inters}, so assume that
	$r:=\rk\cE\geq2$.
	
	Fix an ample divisor $B$ on $T$.
	By \cite[Theorem~4.1]{La-approximation}, there exist
	finite surjective morphisms $f_m:T_m\to T$ from smooth
	projective surfaces and filtrations of $f_m^*\cE$ by
	subbundles with line bundle quotients
	$\cL_{1,m},\ldots,\cL_{r,m}$ such that, setting
	\[
	\delta_m:=\deg f_m,
	\qquad N_{i,m}:=c_1(\cL_{i,m}),
	\]
	we have
	\[
	\frac{(f_m)_*N_{i,m}}{\delta_m}
	\equiv\frac{\gamma_i}{m}B
	\]
	for rational numbers $\gamma_1,\ldots,\gamma_r$
	independent of $m$.
	The filtration and numerical flatness give
	\[
	\sum_iN_{i,m}\equiv0,
	\qquad
	\sum_{i<j}N_{i,m}\cdot N_{j,m}
	=\int_{T_m}c_2(f_m^*\cE)=0.
	\]
	Consequently,
$\sum_iN_{i,m}^2=0.$

	Consider the Stein factorization
	\[
	T_m\xrightarrow{\rho_m}S_m\longrightarrow S
	\]
	of $\pi f_m$.
	Let $D_{i,m}$ be the exceptional $\mathbb Q$-divisor
	representing the first relative Chern class of
	$\cL_{i,m}$ with respect to $\rho_m$, summed over its
	exceptional fibers.
	Since $\det f_m^*\cE$ is numerically trivial, negative
	definiteness of the exceptional intersection matrices gives
	\[
	\sum_iD_{i,m}=0.
	\]
	By the defining property of $D_{i,m}$, we also have
	\[
	N_{i,m}\cdot D_{i,m}=D_{i,m}^2.
	\]
	
	Choose an ample divisor $H_S$ on $S$, and put
	$P=\pi^*H_S$ and $K_m=f_m^*P$.
	Then $K_m$ is nef and big, it is orthogonal to every
	$\rho_m$-exceptional curve, and
	\[
	K_m^2=\delta_mP^2,
	\qquad
	N_{i,m}\cdot K_m
	=\frac{\delta_m\gamma_i}{m}(B\cdot P).
	\]
	Applying the Hodge index theorem to $N_{i,m}-D_{i,m}$
	gives
	\[
	0\leq-D_{i,m}^2
	\leq
	\frac{(N_{i,m}\cdot K_m)^2}{K_m^2}
	-N_{i,m}^2.
	\]
	Summing over $i$ and using $\sum_iN_{i,m}^2=0$, we obtain
	\[
	0\leq-\sum_iD_{i,m}^2
	\leq
	\frac{\delta_m(B\cdot P)^2}{m^2P^2}
	\sum_i\gamma_i^2.
	\]
	
	The filtration of $f_m^*\cE$ is admissible in the
	definition of the relative second Chern class.
	Applying \cite[Definition~3.1]{La-Inters} over the
	henselizations at the exceptional points and summing gives
	\[
		0\leq\sum_xc_2(\pi_x,\cE)
		\leq
		\frac{1}{\delta_m}
		\sum_{i<j}D_{i,m}\cdot D_{j,m}
		=-\frac{1}{2\delta_m}\sum_iD_{i,m}^2
		\leq
		\frac{(B\cdot P)^2}{2m^2P^2}\sum_i\gamma_i^2.
	\]
	Here, if the base change splits into several components,
	one applies the defining inequality to each component,
	multiplies by its generic degree, and adds; these degrees
	sum to $\delta_m$.
	Letting $m$ tend to infinity proves the required vanishing.
	
	Finally, \cite[Definition~4.1]{La-Inters} gives
	\[
	\int_Sc_2(\cG)
	=
	\int_Tc_2(\cE)-\sum_xc_2(\pi_x,\cE)=0.\]
	Together with $c_1(\cG)\equiv0$, this implies
	$\int_S\ch_2(\cG)=0$.
\end{proof}

\begin{Theorem}\label{nf-normal:prop-weak-positivity}
	Let $\overline X$ be a normal projective variety of dimension $n$
	over an algebraically closed field $k$. Fix an ample divisor $H$
	on $\overline X$.
	Let $j:X\hookrightarrow\overline X$ be a big open subset, and let
	$\cE$ be a vector bundle on $X$. Put $\cF=j_*\cE$.
	Assume that either $\overline X$ is smooth or $\cF$ is locally
	free on $\overline X$.
	Then the following conditions are equivalent:
	\begin{enumerate}
		\item $\cE$ is strongly numerically flat.
		\item $\cE$ is weakly flat.
		\item $\cF$ is strongly slope $H$-semistable and
		\[
		c_1(\cF)\cdot H^{n-1}=0,
		\qquad
		\int_{\overline X}\ch_2(\cF)H^{n-2}=0.
		\]
		\item $\cF$ is a numerically flat vector bundle on $\overline X$.
	\end{enumerate}
\end{Theorem}

\begin{proof}
	Assume (4). Since nef bundles on a projective variety are weakly
	positive, restriction proves (2). 
	To prove (1), let $i:S\hookrightarrow\overline S$ be a small
	compactification of a normally big surface and let $g:S\to X$.
	Resolve the rational map $\overline S\dashrightarrow\overline X$
	induced by $jg$. This gives a smooth projective surface $T$ and
	morphisms
	\(
	\pi:T\to\overline S,\) \(h:T\to\overline X,
	\)
	with $\pi$ projective and birational and
	$h|_{\pi^{-1}(S)}=jg\pi$. The bundle $\cV=h^*\cF$ is numerically
	flat. The projection formula over $S$ and normality give
	\[
	(\pi_*\cV)|_S=g^*\cE,\qquad
	(\pi_*\cV)^{**}=i_*g^*\cE.
	\]
	Lemma~\ref{lem:nf-birational-chern} proves the required second
	Chern character vanishing. Numerical flatness of $\cE$ is immediate.
	
	We next deduce (4) from either (1) or (2). In both cases
	$\cE$ is numerically flat, by definition or by
	Lemma~\ref{snf-wp:elementary-properties}. If $n\le1$, then
	$X=\overline X$ and there is nothing more to prove.
	Assume henceforth that $n\ge2$.
	
	Suppose first that $\overline X$ is smooth. For any ample divisor
	$H$ on $\overline X$, restriction to general complete intersection
	curves contained in $X$ shows that $\cF$ is strongly slope
	$H$-semistable and $c_1(\cF)\cdot H^{n-1}=0.$
	For semistability, restrict a hypothetical destabilizing subsheaf
	to a general curve avoiding the non-locally-free loci involved; 
	apply the same argument to each Frobenius pullback when $p>0$.
	Choose a general smooth complete intersection surface
	$T\subset\overline X$ avoiding the non-locally-free locus of the
	reflexive sheaf $\cF$. Then $T\cap X$ is big in $T$ and
	$\cF|_T$ is locally free. Under (1), the defining surface condition
	gives $\int_T\ch_2(\cF|_T)=0$.
	Under (2), restrict the sections witnessing weak positivity to
	$T$. Lemma~\ref{surface-weak-positivity} shows that $\cF|_T$ and
	its inverse determinant are nef. Thus $\cF|_T$ is numerically
	flat, and again its second Chern character vanishes. In either case,
	$	\int_{\overline X}\ch_2(\cF)H^{n-2}=0.$
	The smooth-projective criterion \cite[Theorem~B]{FL} now shows
	that $\cF$ is locally free and numerically flat.
	
	We may therefore assume that $\overline X$ is normal and $\cF$
	is locally free. If $\operatorname{char}k=p>0$,
	Theorem~\ref{snf-wp:equivalence} implies that $\cE$ is
	strongly numerically flat under either (1) or (2). Then by \cite[Corollary 4.6]{La-Simpson}
	the families
	\[
	\{(F_{\overline X}^a)^*\cF\}_{a\ge0},\qquad
	\{(F_{\overline X}^a)^*\cF^*\}_{a\ge0}
	\]
	are bounded. So for any morphism $f:C\to\overline X$ from a smooth
	projective curve, the families
	\[
	\{(F_{C}^a)^*(f^*\cF) \}_{a\ge0},\qquad
	\{(F_{C}^a)^*(f^*\cF^*) \}_{a\ge0}
	\]
are bounded. So $\cF$ is numerically flat (e.g., again by \cite[Corollary 4.6]{La-Simpson}), proving (4).
	
	Assume now that $\operatorname{char}k=0$.
	Recall that $\cF$ is locally free on $\overline X$.
	We show that either (1) or (2) implies (3).
	Under either assumption, $\cE$ is numerically flat on $X$.
	As in the smooth case, restriction to general complete
	intersection curves contained in $X$ shows that $\cF$ is
	slope $H$-semistable and
	\(
	c_1(\cF)\cdot H^{n-1}=0.
	\)
	Choose a sufficiently large integer $m$ and a general normal
	complete intersection surface
	\[
	T=D_1\cap\cdots\cap D_{n-2},
	\qquad D_i\in|mH|,
	\]
	such that $T^\circ:=T\cap X$ is big in $T$.
	For $n=2$, take $T=\overline X$.
	Put $\cV:=\cF|_T$.
	
	Under (1), the defining surface condition for strong numerical
	flatness gives
	\[
	\int_T\ch_2(\cV)=0,
	\]
	since $\cV$ is the reflexive extension of
	$\cE|_{T^\circ}$.
	
	Under (2), restriction of the sections witnessing weak
	positivity shows that both $\cV$ and $\cV^*$ are weakly
	positive over $T^\circ$.
	Lemma~\ref{surface-weak-positivity} therefore shows that
	both bundles are nef on $T$.
	Thus $\cV$ is numerically flat, and again
	$\int_T\ch_2(\cV)=0$.
	
	Consequently, in either case,
	\[
	0=\int_T\ch_2(\cV)
	 =m^{n-2}\int_{\overline X}\ch_2(\cF)H^{n-2}.
	\]
	This proves (3).
	
	Condition (4) also implies (3), since numerically flat
	bundles are strongly slope semistable and have numerically
	trivial Chern classes.
	To complete the proof, it therefore remains to establish
	$(3)\Rightarrow(4)$.
	If $\overline X$ is smooth, this follows from
	\cite[Theorem~B]{FL}.

	If $\operatorname{char}k>0$ and $\cF$ is locally free then  equivalence of (3) with the remaining conditions follows from Theorem \ref{snf-wp:equivalence}.
 	
	Assume that $\operatorname{char}k=0$ and that $\cF$ is locally
	free. We argue by induction on $n=\dim\overline X$. The curve
	case is immediate, and the surface case is
	Lemma~\ref{nf-normal:lem-surface-char-zero}. 
	
	Suppose that $n\ge3$. By the restriction theorem, for sufficiently
	large $m$ and a general normal divisor $D\in|mH|$, the bundle
	$\cF|_D$ is slope $H|_D$-semistable and satisfies the same
	Chern-number vanishings. Hence $\cF|_D$ is numerically flat by
	induction.

	Let $S\subset\overline X$ be any integral surface, let
	$\nu:S^\nu\to S$ be its normalization, and set
	\[
	\cV_S:=\nu^*(\cF|_S),\qquad H_S:=\nu^*(H|_S).
	\]
	Choosing $D$ sufficiently general, the restriction of $\cV_S$
	to the smooth curve $\nu^{-1}(S\cap D)$ is numerically flat.
	Restricting a hypothetical destabilizing subsheaf to such a
	curve shows that $\cV_S$ is slope $H_S$-semistable, with
	$c_1(\cV_S)\cdot H_S=0$. Bogomolov's inequality and the Hodge
	index theorem therefore give
	\[
	\int_{S^\nu}\ch_2(\cV_S)
	\le \frac{c_1(\cV_S)^2}{2\rk\cF}\le0.
	\]
	Thus $\ch_2(\cF)$ pairs nonpositively with every integral
	surface in $\overline X$.
	
	Choose $(n-2)$ sufficiently high degree hypersurfaces containing $S$
	and meeting properly. Their intersection cycle has the form
	$a[S]+\sum_j a_j[S_j]$ for some  $a,a_j>0,$
	and its pairing with $\ch_2(\cF)$ is zero by (3).
	Since every summand has nonpositive pairing, we obtain
	\[
	\int_{S^\nu}\ch_2(\cV_S)=0.
	\]
	Lemma~\ref{nf-normal:lem-surface-char-zero} now shows that
	$\cV_S$ is numerically flat for every integral surface $S$.
	
	Let $C\subset\overline X$ be an integral curve, and choose
	an integral surface $S\subset\overline X$ containing $C$.
	Write $\nu:S^\nu\to S$ for its normalization.
	An irreducible component of $\nu^{-1}(C)$ dominates $C$.
	Its normalization $C'$ therefore admits a finite surjective
	morphism
	\(
	q:C'\longrightarrow C^\nu,
	\)
	where $C^\nu$ is the normalization of $C$.
	
	By the preceding paragraph, $\nu^*(\cF|_S)$ is numerically
	flat. Its pullback to $C'$ is consequently numerically
	flat, and this pullback is $q^*(\cF|_{C^\nu})$.
	Numerical flatness descends under finite surjective
	morphisms, so $\cF|_{C^\nu}$ is numerically flat.
	The curve criterion now shows that $\cF$ is numerically
	flat on $\overline X$, proving (4).
\end{proof}

\medskip

\begin{Remark}
The Chern-number conditions in (3) in Theorem \ref{nf-normal:prop-weak-positivity} are imposed only in dimensions
	where they are defined.
	Note also that if $\overline X$ is smooth, local freeness of
	$\cF$ is a conclusion and need not be assumed.
\end{Remark}

\begin{Remark}
Assume that $\bar X$ is singular and $\cF$ is not assumed to be locally free. Then (4) does not follow from the remaining conditions (see \cite[Example 4.14]{La-Simpson}). 
In positive characteristic Theorem \ref{snf-wp:equivalence} implies that conditions (1)--(3) are equivalent. In characteristic zero, condition (3) is not defined in this
generality in higher dimensions.
In this case Corollary~\ref{cor:weak-flat-implies-snf-char-zero}
gives $(2)\Rightarrow(1)$, and on surfaces we obtain
$(2)\Rightarrow(1)\Rightarrow(3)$. For arbitrary normal projective surfaces, the converse implication $(3)\Rightarrow(2)$  is not established.
\end{Remark}
 
\subsection{Frobenius pushforwards and ordinary abelian varieties}

We first record the following elementary observation concerning the
Frobenius pushforward of the de Rham differential.

\begin{Lemma}\label{lem:frobenius-cotangent-quotient}
Let $X$ be a smooth variety over a perfect field $k$ of
characteristic $p>0$. For every integer $e\geq1$, the adjoint
of the $\cO_X$-linear homomorphism
\(
(F_X^e)_*d:
(F_X^e)_*\cO_X
\longrightarrow
(F_X^e)_*\Omega^1_{X/k}
\)
is a surjective homomorphism
\[
\vartheta_e:
(F_X^e)^*(F_X^e)_*\cO_X
\twoheadrightarrow
\Omega^1_{X/k}.
\]
\end{Lemma}

\begin{proof}
Locally, the adjoint homomorphism is given by
\(
\vartheta_e(a\otimes b)=a\,db.
\)
Choose local functions $t_1,\ldots,t_n$ whose differentials
form a basis of $\Omega^1_{X/k}$. Since
\(
\vartheta_e(1\otimes t_i)=dt_i,
\)
the homomorphism $\vartheta_e$ is surjective.
\end{proof}

\begin{Corollary}\label{cor:frobenius-weak-positivity-abelian}
Let $X$ be a smooth projective globally $F$-split variety
of positive dimension over an algebraically closed field
$k$ of characteristic $p>0$, and let $U\subseteq X$ be
a big open subset. Then the following conditions are
equivalent:
\begin{enumerate}
    \item For some integer $e\geq1$, the vector bundle
    \(
    (F_X^e)_*\cO_X
    \)
    is weakly positive over $U$.

    \item The vector bundle $\Omega^1_{X/k}$ is
    weakly positive over $U$.

    \item There exist an ordinary abelian variety $A$
    and a finite surjective \'etale morphism
    $\pi:A\to X$.

    \item For every integer $e\geq1$, the vector bundle
    $(F_X^e)_*\cO_X$ is numerically flat on $X$.
\end{enumerate}
If these conditions hold, then
\(
\omega_X^{\otimes(p-1)}\simeq\cO_X.
\)
\end{Corollary}

\begin{proof}
Assume (1). Weak positivity is preserved by pullback
and by locally free quotients. Since
$(F_X^e)^{-1}(U)=U$, Lemma
\ref{lem:frobenius-cotangent-quotient} implies that
$\Omega^1_{X/k}|_U$ is weakly positive  over $U$. Thus (1)
implies (2).

Assume (2), and put $n=\dim X$. Choose a very ample
divisor $H$ on $X$. Since $X\setminus U$ has codimension
at least two, a general complete intersection curve
of sufficiently ample multiples of $H$ is contained
in $U$. Restriction to this curve shows that
\(
K_X\cdot H^{n-1}\geq0.
\)
On the other hand, global $F$-splitting and duality
for the finite morphism $F_X$ give a nonzero section
\(
s\in H^0(X,\omega_X^{\otimes(1-p)}).
\)
Its zero divisor $D$ is effective and satisfies
\[
D\cdot H^{n-1}
=(1-p)K_X\cdot H^{n-1}\leq0.
\]
As $H$ is ample, this forces $D=0$. Consequently,
\(
\omega_X^{\otimes(p-1)}\simeq\cO_X.
\)
In particular, $\omega_X^{-1}$ is weakly positive  over $U$.
Lemma~\ref{snf-wp:weak-flat-determinant} therefore
shows that $\Omega^1_{X/k}$ is weakly flat  over $U$.
Since
\[
j_*(\Omega^1_{X/k}|_U)=\Omega^1_{X/k},
\qquad j:U\hookrightarrow X,
\]
Theorem~\ref{nf-normal:prop-weak-positivity} implies
that $\Omega^1_{X/k}$ is numerically flat.
For $n=1$, this follows directly from the torsion
property of $\omega_X$.

Thus $T_X$ is numerically flat. Since $X$ is also globally $F$-split,   
the theorem of
Ejiri--Yoshikawa \cite[Theorem 1.2]{Ejiri-Yoshikawa} gives a finite
surjective \'etale morphism $A\to X$ from an ordinary
abelian variety. This proves (2)$\Rightarrow$(3).

Assume (3), and fix $e\geq1$. Since $\pi$ is \'etale,
the square formed by $\pi$, $F_A^e$, and $F_X^e$
is cartesian. Finite flat base change gives
\[
\pi^*(F_X^e)_*\cO_X
\simeq
(F_A^e)_*\cO_A.
\]
The relative Frobenius of an abelian variety is a
torsor under its finite kernel. Hence its
self-base-change is a trivial torsor, and consequently
\[
(F_A^e)^*(F_A^e)_*\cO_A
\simeq
\cO_A^{\oplus p^{e\dim A}}.
\]
Numerical flatness descends under finite surjective
pullback. It follows first that $(F_A^e)_*\cO_A$
is numerically flat, and then that $(F_X^e)_*\cO_X$
is numerically flat. This proves (3)$\Rightarrow$(4).

Finally, a numerically flat bundle on a projective
variety is weakly positive, and its restriction to
$U$ is weakly positive over $U$. Thus (4) implies (1).
\end{proof}

\medskip

We can partially generalize the above corollary to the normal projective case:

\begin{Proposition}\label{prop:frobenius-weak-positivity-normal}
Let $X$ be a normal projective globally $F$-split variety
of dimension $n\geq2$ over an algebraically closed field
$k$ of characteristic $p>0$. Let
$U\subseteq X_{\mathrm{reg}}$ be a big open subset.
Assume that, for some integer $e\geq1$, the reflexive sheaf
\(
(F_X^e)_*\cO_X
\)
is weakly positive over $U$. Then
\(
\omega_X^{[p-1]}\simeq\cO_X,
\)
and both $\Omega^1_{U/k}$ and $(F_X^e)_*\cO_X|_U$ are weakly flat.
In particular, they are strongly numerically flat.

If, moreover, there exists a finite surjective
quasi-\'etale morphism $f:Y\to X$ with $Y$ smooth,
then there exist an ordinary abelian variety $A$
and a finite surjective quasi-\'etale morphism $A\to X$.
\end{Proposition}

\begin{proof}
On $U$, Lemma~\ref{lem:frobenius-cotangent-quotient}
gives a surjection
\(
(F_U^e)^*(F_U^e)_*\cO_U
\twoheadrightarrow \Omega^1_{U/k}.
\)
Since weak positivity is preserved by pullback and
locally free quotients, $\Omega^1_{U/k}$ is weakly
positive over $U$.

Choose a very ample divisor $H$ on $X$.
Restriction to a general complete intersection curve
contained in $U$ gives
\(
K_X\cdot H^{n-1}\geq0.
\)
The restriction of a global Frobenius splitting to $U$,
together with finite duality on $U$, determines a
nonzero section of $\omega_U^{\otimes(1-p)}$.
Since $X$ is normal and $U$ is big, this extends to
a nonzero section
\(
s\in H^0(X,\omega_X^{[1-p]}).
\)
Its associated effective Weil divisor $D$ satisfies
\[
D\sim(1-p)K_X,
\qquad
D\cdot H^{n-1}
=(1-p)K_X\cdot H^{n-1}\leq0.
\]
An effective nonzero Weil divisor has positive
intersection with $H^{n-1}$. Hence $D=0$, and therefore
\(
\omega_X^{[p-1]}\simeq\cO_X.
\)
In particular, $\omega_U^{-1}$ is torsion and hence
weakly positive. Lemma~\ref{snf-wp:weak-flat-determinant}
shows that $\Omega^1_{U/k}$ is weakly flat.

Put $q=p^e$. Since $(p-1)$ divides $(q-1)$, we have
\(
\omega_U^{\otimes(1-q)}\simeq\cO_U.
\)
So finite duality for Frobenius on the smooth
variety $U$ gives
\[
\bigl((F_U^e)_*\cO_U\bigr)^\vee
\simeq
(F_U^e)_*\omega_U^{\otimes(1-q)}\simeq(F_U^e)_*\cO_U.
\]
Thus $(F_X^e)_*\cO_X|_U$  is also weakly flat.
Theorem~\ref{snf-wp:equivalence} now implies strong
numerical flatness of both bundles.

Finally, suppose that $f:Y\to X$ is finite,
surjective and quasi-\'etale, with $Y$ smooth.
By purity, $f$ is \'etale over $X_{\mathrm{reg}}$.
Set $V=f^{-1}(U)$.
Then
\[
\Omega^1_{Y/k}|_V
\simeq
(f|_V)^*\Omega^1_{U/k}
\]
is weakly positive over $V$.

The Frobenius splitting of $U$ pulls back along the
\'etale morphism $V\to U$ to a splitting of $V$, since
the corresponding Frobenius square is cartesian.
As $Y$ is normal and
$\operatorname{codim}_Y(Y\setminus V)\geq2$, this splitting
extends uniquely to a Frobenius splitting of $Y$ by
\cite[Lemma~1.1.7]{Brion-Kumar}.
Hence $Y$ is globally $F$-split.

Applying Corollary~\ref{cor:frobenius-weak-positivity-abelian}
to $Y$ and $V$ gives a finite surjective \'etale morphism
$A\to Y$ from an ordinary abelian variety.
Its composition with $f$ is the required finite
surjective quasi-\'etale morphism $A\to X$.
\end{proof}

\begin{Remark}
We expect that if $X$
is globally $F$-split and strongly $F$-regular,
and $\Omega^1_{X/k}$ is weakly flat  over
$X_{\mathrm{reg}}$, then there exist an ordinary abelian
variety $A$ and a finite surjective quasi-\'etale morphism
\(
\pi:A\to X.
\) Together with the above results, this would lead to  an analogue of Theorem \ref{main4} in positive characteristic.
\end{Remark}

\section{Weak positivity and unitary flatness on normally big varieties in characteristic zero}
\label{sec:unitary-weak-positivity}

\subsection{Small compactifications with klt singularities}

In this subsection the ground field is algebraically closed of
characteristic zero. The following result generalizes the case of
Theorem~\ref{nf-normal:prop-weak-positivity} in which $\overline X$ is smooth.
Below hats denote orbifold Chern classes. For varieties with quotient
singularities in codimension $2$, the corresponding Chern numbers
coincide with those defined in \cite{La-Inters}. They can be computed by restriction to sufficiently general
complete intersection surfaces, using the surface Chern numbers
introduced earlier in the paper.

\begin{Theorem}\label{snf-wp:equivalence-klt}
	Let $X$ be a smooth variety of dimension
	$n\ge2$. Assume that $X$ admits a small projective
	compactification $j:X\hookrightarrow\overline X$ with
	klt singularities.
	Fix an ample line bundle $H$ on $\overline X$, let $\cE$
	be a vector bundle on $X$, and set $\cF:=j_*\cE$.
	Then the following conditions are equivalent:
	\begin{enumerate}
		\item $\cE$ is strongly numerically flat.
		\item $\cE$ is weakly flat.
		\item $\cF$ is slope $H$-semistable and
		\[
		\widehat c_1(\cF)\cdot H^{n-1}=0,
		\qquad
		\int_{\overline X}\widehat{\ch}_2(\cF)H^{n-2}=0.
		\]
		\item There exists a finite surjective Galois morphism
		$\gamma:Y\to\overline X$, \'etale over
		$\overline X_{\mathrm{reg}}$, such that
		$\gamma^{[*]}\cF$ is a numerically flat vector bundle
		on $Y$.
	\end{enumerate}
\end{Theorem}

\begin{proof}
	The implication $(3)\Rightarrow(4)$ follows from
	\cite[Theorem~1.4]{LT}. The reflexive pullback on the
	cover constructed there is locally free and admits a
	filtration with unitary flat quotients, hence it is
	numerically flat. Although the result is formulated
	over $\mathbb C$, standard spreading out and base-change
	arguments give the assertion over any algebraically
	closed field of characteristic zero.
	We prove $(1)\Rightarrow(3)$,
	$(2)\Rightarrow(3)$, and $(4)\Rightarrow(1),(2)$.
	
	Assume either (1) or (2). In both cases $\cE$ is
	numerically flat, by definition or by
	Lemma~\ref{snf-wp:elementary-properties}.
	Restricting to general complete intersection curves
	contained in $X$ gives
	\begin{equation}\label{snf-wp:klt-first-chern}
		\cF\text{ is slope }H\text{-semistable},
		\qquad \widehat c_1(\cF)\cdot H^{n-1}=0.
	\end{equation}
	Indeed, a destabilizing subsheaf would restrict to a
	positive-degree subsheaf of a semistable bundle of
	degree zero.
	
	Choose a sufficiently general complete intersection
	surface $S\subset\overline X$, cut out by sufficiently
	high multiples of $H$, such that $S$ is normal with
	quotient singularities, $S^\circ:=S\cap X$ is big in
	$S$, and $\cF_S:=\cF|_S$ is reflexive.
	If $n=2$, take $S=\overline X$.
	Under (1), the defining surface condition gives
	$\int_S\widehat{\ch}_2(\cF_S)=0$.
	
	Under (2), \cite[Lemma~7.1]{La-Bog-normal} provides a
	finite surjective morphism $\pi:T\to S$ from a normal
	projective surface such that
	$\mathcal W:=\pi^{[*]}\cF_S$ is locally free.
	Both $\mathcal W$ and $\det\mathcal W^*$ are weakly
	positive over $\pi^{-1}(S^\circ)$, hence nef on $T$
	by Lemma~\ref{surface-weak-positivity}.
	Thus $\mathcal W$ is numerically flat, and
	\[
	0=\int_T\ch_2(\mathcal W)
	=\deg(\pi)\int_S\widehat{\ch}_2(\cF_S).
	\]
	In either case, compatibility with complete
	intersection restriction gives
	\begin{equation*}
		\int_{\overline X}\widehat{\ch}_2(\cF)H^{n-2}=0.
	\end{equation*}
	Together with \eqref{snf-wp:klt-first-chern}, this
	proves (3).
	
	Now assume (4), and set
	$\cV:=\gamma^{[*]}\cF$ and $Y_X:=\gamma^{-1}(X)$.
	Then $\cV|_{Y_X}=\gamma|_{Y_X}^*\cE$.
	Since $\cV$ and $\cV^*$ are nef on the projective
	variety $Y$, they are weakly positive.
	Restriction and finite descent of weak positivity
	give $\cE\in\Vect^{\wf}(X)$, proving (2).
	
	Numerical flatness of $\cE$ also follows by finite
	descent. To prove (1), let $g:V\to X$ be a morphism
	from a normally big surface, let
	$i:V\hookrightarrow\overline V$ be its small
	compactification, and put $\cG=i_*g^*\cE$.
	Normalize a component of $V\times_XY_X$ dominating
	$V$, obtaining a finite surjective morphism
	$q:V'\to V$ and a morphism $h:V'\to Y_X$.
	Normalizing $\overline V$ in $k(V')$ extends $q$ to a
	finite morphism
	$\overline q:\overline V'\to\overline V$, with $V'$
	big in $\overline V'$.
	Write $i':V'\hookrightarrow\overline V'$.	
	By Theorem~\ref{nf-normal:prop-weak-positivity}, 
	the numerically flat bundle $\cV$ on the projective
	variety $Y$ is strongly numerically flat.
	Apply its defining surface condition to the composite
	$V'\xrightarrow{h}Y_X\hookrightarrow Y$.
	Since
	\[
	i'_*h^*(\cV|_{Y_X})\simeq\overline q^{[*]}\cG,
	\]
	\cite[Proposition~4.2(4)]{La-Inters} gives
	\[
	0=\int_{\overline V'}\ch_2(\overline q^{[*]}\cG)
	=\deg(\overline q)\int_{\overline V}\ch_2(\cG).
	\]
	This proves (1) and completes the proof.
\end{proof}

\begin{Remark}
	The finite Galois cover in
	Theorem~\ref{snf-wp:equivalence-klt}(4) can be chosen
	independently of $\cE$.
	Indeed, \cite[Theorem~1.4]{LT} provides a single cover
	depending only on $\overline{X}$.
\end{Remark}

\begin{Corollary}\label{cor:klt-filtered-connection}
	Let $X$ be a complex projective klt variety of dimension
	$n\ge2$, and let $j:U\hookrightarrow X$ be a smooth
	big open subset.
	For a vector bundle $\cE$ on $U$, the following conditions
	are equivalent:
	\begin{enumerate}
		\item $\cE$ is weakly flat.
		\item $\cE$ admits an integrable algebraic connection
		$\nabla$ and a filtration
		\[
		0=\cE_0\subset\cE_1\subset\cdots\subset\cE_s=\cE
		\]
		by $\nabla$-invariant subbundles such that the induced
		connections on $\cE_i/\cE_{i-1}$ have unitary monodromy.
	\end{enumerate}
\end{Corollary}

\begin{proof}
	Assume (1), put $\cF:=j_*\cE$, and fix an ample line
	bundle $H$ on $X$.
	By Theorem~\ref{snf-wp:equivalence-klt}, $\cF$ is slope
	$H$-semistable, $\widehat c_1(\cF)\cdot H^{n-1}=0$ and	$\widehat{\ch}_2(\cF)\cdot H^{n-2}=0.$
	The construction in
	\cite[Theorem~1.4 and Section~4.B]{LT} gives an integrable
	holomorphic connection on
	$(\cF|_{X_{\mathrm{reg}}})^{\an}$ preserving a filtration
	with unitary flat quotients.
	Restricting to $U$, Hartogs extension and GAGA algebraize
	the connection and its invariant filtration.
	This proves (2).
	Conversely, (2) implies (1) by
	Proposition~\ref{prop:unitary-uniform-generation}.
\end{proof}

\medskip

Combining Theorem~\ref{snf-wp:equivalence-klt} with
\cite[Theorem~1.2]{GKP-projectively-flat}, we obtain the
following characterization.

\begin{Corollary}\label{cor:normalized-cotangent-weak-positivity}
	Let $X$ be a complex projective klt variety of dimension
	$n\ge2$.
	Then the following conditions are equivalent:
	\begin{enumerate}
		\item The vector bundle
		\(
		\Sym^n\Omega_{X_{\mathrm{reg}}}^1\otimes\omega_{X_{\mathrm{reg}}}^{-1}
		\)
		is weakly positive over $X_{\mathrm{reg}}$.
		\item There exist an abelian variety $A$ and a finite
		surjective quasi-\'etale morphism $A\to X$. 
	\end{enumerate}
\end{Corollary}

\begin{proof}
	Let $j:U:=X_{\mathrm{reg}}\hookrightarrow X$ be the inclusion, fix an ample
	line bundle $H$ on $X$, and set
	\[
	\cB_U:=\Sym^n\Omega_U^1\otimes\omega_U^{-1},
	\qquad
	\cB:=j_*\cB_U,
	\qquad
	N:=\rk\cB=\binom{2n-1}{n-1}.
	\]

	Assume (1).  Since $\det\cB_U\simeq\cO_U$, Lemma~\ref{snf-wp:weak-flat-determinant} shows that
	$\cB_U$ is weakly flat.
	Theorem~\ref{snf-wp:equivalence-klt} implies that
	$\cB$ is slope $H$-semistable and
	\(
	\widehat{\ch}_2(\cB)\cdot H^{n-2}=0.
	\)

	We claim that $\Omega_X^{[1]}$ is slope $H$-semistable.
	Otherwise, there is a nonzero proper saturated subsheaf
	$\cG\subset\Omega_X^{[1]}$ with
	\(
	\mu_H(\cG)>\mu_H(\Omega_X^{[1]}).\)
	The induced inclusion on a big open subset extends to
	an inclusion
	\[
	\cA:=
	\bigl(\Sym^n\cG\otimes\cO_X(-K_X)\bigr)^{**}
	\hookrightarrow\cB.
	\]
	However,
	\[
	\mu_H(\cA)
	=
	n\bigl(\mu_H(\cG)-\mu_H(\Omega_X^{[1]})\bigr)>0
	=
	\mu_H(\cB),
	\]
	contradicting semistability of $\cB$.
	
	The splitting principle for orbifold Chern classes gives
	\[
	\widehat{\ch}_2(\cB)\cdot H^{n-2}
	=
	-\frac{N}{n+1}
	\left(
	2n\,\widehat c_2(\Omega_X^{[1]})
	-(n-1)\widehat c_1(\Omega_X^{[1]})^2
	\right)\cdot H^{n-2}.
	\]
	Consequently,
	\[
	\left(
	2n\,\widehat c_2(\Omega_X^{[1]})
	-(n-1)\widehat c_1(\Omega_X^{[1]})^2
	\right)\cdot H^{n-2}=0.
	\]
	Together with semistability of $\Omega_X^{[1]}$, this
	verifies the hypotheses of
	\cite[Theorem~1.2]{GKP-projectively-flat}, which yields (2).
	
	Conversely, assume (2), and let $\gamma:A\to X$ be such
	a morphism. Put $V:=\gamma^{-1}(U)$.
	By purity of the branch locus, the restriction
	$f:=\gamma|_V:V\to U$ is finite \'etale.
	Hence
	\[
	f^*\Omega_U^1\simeq\Omega_A^1|_V
	\simeq\cO_V^{\oplus n},
	\]
	and therefore $f^*\cB_U$ is trivial.
	Finite descent of weak positivity implies that $\cB_U$
	is weakly positive over $U$, proving (1).
\end{proof}

\medskip

Set $\cB_X:=
\bigl(\Sym^n\Omega_X^1\otimes\cO_X(-K_X)\bigr)^{**}.$
If $\cB_X$ is nef in the sense of
\cite{GKP-projectively-flat}, then condition~(1) holds.
Thus the corollary strengthens the criterion in
\cite[Theorem~1.3]{GKP-projectively-flat} by replacing
nefness of $\cB_X$ on $X$ with weak positivity of its
restriction over $X_{\mathrm{reg}}$.
This weakening is strict, as the following example shows.

\begin{Example}\label{ex:normalized-cotangent-not-nef}
	Let $E$ be a complex elliptic curve with an automorphism
	$\sigma$ of order three fixing the origin. Set
	\[
	A:=E\times E,\qquad
	\tau:=(\sigma,\sigma^{-1}),\qquad
	X:=A/\langle\tau\rangle,
	\]
	and let $\pi:A\to X$ be the quotient morphism.
	Then $\pi$ is quasi-\'etale, $X$ has only singularities
	of type $A_2$, and $\omega_X\simeq\cO_X$.
	In particular, condition~(2) of
	Corollary~\ref{cor:normalized-cotangent-weak-positivity}
	holds.
	
	The representation of $\langle\tau\rangle$ on
	$\Sym^2H^0(A,\Omega_A^1)$ contains each of its three
	characters exactly once. Taking invariants on the
	\'etale locus and extending reflexively therefore gives
	\[
	\cB_X=(\Sym^2\Omega_X^1)^{**}
	\simeq\pi_*\cO_A.
	\]
	
	Let $C:=\pi(E\times\{0\})$. Then
	$C\simeq E/\langle\sigma\rangle\simeq\PP^1$.
	Writing $q:E\to C$ for the quotient morphism,
	restriction of functions induces a surjection
	\[
	\cB_X|_C
	\simeq(\pi_*\cO_A)|_C
	\twoheadrightarrow q_*\cO_E.
	\]
	The vector bundle $q_*\cO_E$ has rank three and,
	by the Riemann--Roch theorem, degree $-3$. Hence it is not nef.
	Consequently, $\cB_X$ is not nef, although
	$\cB_X|_{X_{\mathrm{reg}}}$ is weakly positive over
	$X_{\mathrm{reg}}$.
\end{Example}

\subsection{Unitary flat bundles on smooth normally big complex varieties}

In this subsection, $X$ is a smooth  complex variety
of dimension $n>0$, admitting a small normal projective compactification $j:X\hookrightarrow\overline X$, and $H$ is an ample line bundle
on $\overline X$.
A vector bundle on $X$ is called \emph{unitary flat} if its analytification
is induced by a unitary representation of
$\pi_1^{\mathrm{top}}(X^{\mathrm{an}})$.

\begin{Lemma}\label{lem:unitary-log-vanishing}
	Let $Y$ be a smooth complex projective variety of dimension $n$,
	let $D$ be a reduced simple normal crossings divisor, and let
	$\mathbb V$ be a complex unitary local system on $Y\setminus D$.
	Let $\cV$ be Deligne's extension, with residue eigenvalues in
	$[0,1)$, of the corresponding flat vector bundle.
	If $A$ is very ample, then
	\[
	H^i\bigl(Y,\cV\otimes\omega_Y(D)\otimes A^k\bigr)=0
	\qquad\text{for every }i>0\text{ and }k>0.
	\]
\end{Lemma}

\begin{proof}
	Choose a smooth divisor $B\in|A^k|$ such that $D+B$ has simple
	normal crossings, and set $W=Y\setminus(D+B)$.
	The canonical extension of $\mathbb V|_W$ is still $\cV$,
	since the monodromy around $B$ is trivial and its residue is
	zero. Timmerscheidt's degeneration theorem for unitary local
	systems \cite{Tim-unitary} gives an $E_1$-degenerate spectral
	sequence
	\[
	E_1^{p,q}
	=H^q\bigl(Y,\cV\otimes\Omega_Y^p(\log(D+B))\bigr)
	\ \Longrightarrow\ H^{p+q}(W,\mathbb V|_W).
	\]
	The complement $Y\setminus B$ is affine, and removing the
	effective Cartier divisor $D|_{Y\setminus B}$ preserves
	affineness. Thus $W$ is smooth affine of dimension $n$, so
	it has the homotopy type of a CW complex of real dimension
	at most $n$. Hence $H^{n+i}(W,\mathbb V|_W)=0$ for $i>0$.
	Degeneration gives $E_1^{n,i}=0$, and the assertion follows
	from
	$\Omega_Y^n(\log(D+B))=\omega_Y(D)\otimes A^k$.
\end{proof}

\begin{Proposition}\label{prop:unitary-uniform-generation}
	Let $\cE$ be a vector bundle on $X$ equipped with an
	integrable algebraic connection $\nabla$ and a filtration
	\[
	0=\cE_0\subset\cE_1\subset\cdots\subset\cE_s=\cE
	\]
	by $\nabla$-invariant subbundles. Assume that the induced
	connections on 
	$\cQ_i:=\cE_i/\cE_{i-1}$ are unitary flat.
	Set $\cF:=j_*\cE$.
	There is an integer $c>0$ such that
	\(
	\Sym^{[m]}\cF\otimes H^c
	\)
	is generated by global sections over $X$ for every integer
	$m\geq0$.
	In particular, $\cE\in\Vect^{\wf}(X)$ and
	$\cF\in\Vect^{\wf}(\overline X/X)$.
\end{Proposition} 

\begin{proof}
	Choose a projective log resolution $\pi:Y\to\overline X$
	which is an isomorphism over $X$ and for which
	$D:=Y\setminus\pi^{-1}(X)$ is a simple normal crossings
	divisor. Identify $\pi^{-1}(X)$ with $X$.
	Let $A$ be very ample on $Y$, and set
	\[
	M:=\omega_Y(D)\otimes A^{n+1}.
	\]
	
	For every $m\geq0$, the given filtration induces a
	filtration of $\Sym^m\cE$ by subbundles invariant under
	the induced connection. Its graded pieces are direct
	sums of bundles of the form $\bigotimes_{i=1}^s\Sym^{a_i}\cQ_i$, where $a_i\geq0$ and $\sum_{i=1}^s a_i=m$.
	The induced analytic connections on these bundles are
	unitary flat.
	
	Let $\cV_m$ be Deligne's canonical extension of the
	analytic flat bundle defining $\Sym^m\cE$, with residue
	eigenvalues in $[0,1)$. Such a choice is possible because
	the local monodromy eigenvalues have absolute value one.
	Exactness of canonical extension gives a filtration
	of $\cV_m$ by subbundles whose graded pieces are the
	canonical extensions of the unitary flat graded pieces
	above. We regard these bundles and their filtrations
	as algebraic by GAGA.
	The canonical extension is taken separately for each $m$.
	
	Lemma~\ref{lem:unitary-log-vanishing}, applied to each
	graded piece, and the cohomology exact sequences of
	the filtration give
	\[
	H^i\bigl(Y,\cV_m\otimes M\otimes A^{-i}\bigr)=0
	\qquad(1\leq i\leq n).
	\]
	Indeed, the ample twist occurring in the lemma is
	$A^{n+1-i}$.
	Castelnuovo--Mumford regularity therefore shows that
	$\cV_m\otimes M$ is globally generated, with the same $M$
	for all $m$.
	
	Choose $c>0$ such that $\pi_*(M^{-1})\otimes H^c$ is
	globally generated on $\overline X$. Since $\pi$ is an
	isomorphism over $X$, the evaluation map shows that
	$M^{-1}\otimes\pi^*H^c$ is generated over $X$ by its
	global sections on $Y$. Multiplication by global
	generators of $\cV_m\otimes M$ gives sections generating
	$\cV_m\otimes\pi^*H^c$ over $X$.
	
	Their restrictions are holomorphic sections of
	$\Sym^m\cE\otimes H^c|_X$. Since
	$\overline X\setminus X$ has codimension at least two,
	analytic Hartogs extension for the reflexive sheaf
	$\Sym^{[m]}\cF\otimes H^c$ extends them to $\overline X$,
	and GAGA makes the extended sections algebraic.
	They generate over $X$, proving the assertion.
	
	For any $a>0$, take $b=c$ and $m=ac$ to obtain weak
	positivity of $\cF$ over $X$.
	The dual connection preserves the dual filtration,
	whose graded pieces are the duals of the $\cQ_i$ and
	are again unitary flat after analytification.
	The same argument therefore proves weak positivity
	of $\cF^*$ over $X$.
\end{proof}

\begin{Remark}
In Proposition \ref{prop:unitary-uniform-generation} we can replace $\nabla$ by an integrable holomorphic connection on the analytification $\cE^{\an}$ (and a filtration of $\cE^{\an}$ by holomorphic subbundles). Indeed, on a smooth normally big variety, the holomorphic splitting of the first-jet sequence extends between the reflexive extensions on the normal projective compactification by Hartogs's theorem. Then GAGA makes it algebraic. Similarly, holomorphic subbundles extend by Hartogs's theorem to reflexive sheaves on  the analytification of a normal projective compactification. So by GAGA they are analytifications of algebraic subbundles of $\cE$.
\end{Remark}

\medskip

Unlike numerical flatness, weak flatness need not be preserved by extensions, even when
both the subbundle and the quotient are trivial.

\begin{Example}\label{ex:unitary-filtration-not-weakly-positive}
	Let ${\cE}$ be a rank-two vector bundle on $\PP^2$
	fitting into an exact sequence
	\[
	0\longrightarrow\cO_{\PP^2}\longrightarrow {\cE}
	\longrightarrow\mathcal I_Z\longrightarrow0,
	\]
	where $Z$ is a nonempty zero-dimensional locally complete
	intersection, as in \cite[Example~3.2]{La-Simpson}.
	Set $X=\PP^2\setminus\Supp Z$. Then
	\[
	0\longrightarrow\cO_X\longrightarrow\cE|_X
	\longrightarrow\cO_X\longrightarrow0.
	\]
	Thus $\cE|_X$ has a filtration with unitary flat quotients. On the other hand, 
$\cE|_X$ is not strongly numerically flat so by Theorem \ref{snf-wp:equivalence-klt}, $\cE|_X\notin\Vect^{\wf}(X)$.
This can also be  seen directly as follows.
	Choose a line $L$ through a point of $Z$.
	The section defining
	$\cO_{\PP^2}\to{\cE}$ restricts to a nonzero
	section of ${\cE}|_L$ with a nonempty zero divisor.
	Its saturated image is therefore $\cO_L(d)$ for some
	$d\geq1$. So ${\cE}|_L\simeq \cO_L(d)\oplus \cO_L(-d)$.
	If $\cE|_X$ (or equivalently $\cE$) were weakly positive over $X$, then for
	some $b>0$ the sheaf
	\[
	\Sym^{2b}{\cE}\otimes\cO_{\PP^2}(b)
	\]
	would be generated  over $X$ by its sections. But its restriction to $L$ has the negative-degree quotient
	$\cO_L((1-2d)b)$, a contradiction.  
	
This example shows also that an extension of weakly flat vector bundles on a normally big variety need not be weakly flat.	
\end{Example}

The semisimple objects of $\Vect^{\wf}(X)$ nevertheless admit
the following characterization.

\begin{Proposition}\label{prop:stable-unitary-weak-positivity}
	Let $\cE$ be a vector bundle on $X$.
	Then the following conditions are equivalent:
	\begin{enumerate}
		\item $\cE$ is unitary flat.
		\item $\cE$ is a semisimple object in $\Vect^{\wf}(X)$.
	\end{enumerate}
\end{Proposition}
\begin{proof}
	By Proposition~\ref{Vect-wf-rigid-abelian}, the category
	$\Vect^{\wf}(X)$ is abelian, with exact sequences given by
	exact sequences of vector bundles. Every object has finite
	length, since rank strictly decreases upon passage to a
	proper quotient.
	
	Assume first that $\cE$ is unitary flat. By
	Proposition~\ref{prop:unitary-uniform-generation}, it belongs
	to $\Vect^{\wf}(X)$. Let $\cA\subseteq\cE$ be a subobject.
	Then $\cA$ is a subbundle. For every smooth projective
	complete intersection curve $C\subset X$, weak flatness
	gives
	\[
	\deg(\cA|_C)=0.
	\]
	The degree formula for a holomorphic subbundle of a unitary
	flat bundle therefore shows that $\cA|_C$ is parallel.
	Such curves can be chosen through any point of $X$ with
	any prescribed tangent direction. Hence $\cA$ is parallel,
	and its orthogonal complement gives a holomorphic splitting.
	
	The corresponding holomorphic projector extends, by
	Hartogs's theorem, to an endomorphism of the analytification
	of $j_*\cE$ on $\overline X$. By GAGA this endomorphism
	is algebraic. Thus $\cA\subseteq\cE$ splits in
	$\Vect^{\wf}(X)$. Since every subobject splits and $\cE$
	has finite length, $\cE$ is semisimple.
	
	Conversely, it suffices to prove that every simple object
	$\cA\in\Vect^{\wf}(X)$ is unitary flat. Set $\cB=j_*\cA$.
	Restriction to a general smooth complete intersection
	curve contained in $X$ shows that $\cB$ is slope
	$H$-semistable and
	\[
	c_1(\cB)\cdot H^{n-1}=0.
	\]
	Choose a nonzero saturated $H$-stable subsheaf
	$\cG\subseteq\cB$ of slope zero. It is reflexive,
	and $\cG^*$ is $H$-stable. Weak positivity of $\cA^*$
	and the generically surjective morphism
	$\cB^*\to\cG^*$ show that $\cG^*$ is pseudo-effective;
	see \cite[Remark~2.3(1) and Proposition~2.4]{CDM-flatness}.
	
	By \cite[Proposition~2.5]{CDM-flatness},
	$\mathcal L:=\det\cG^*$ is pseudo-effective. Since
	\[
	c_1(\mathcal L)\cdot H^{n-1}=0,
	\]
	\cite[Proposition~4.7]{CDM-flatness} shows that the first
	Chern class of $\pi^*\mathcal L/\mathrm{tors}$ on a suitable
	resolution $\pi:Y\to\overline X$ is represented by a real
	$\pi$-exceptional divisor. The determinant of
	$\pi^*\cG^*/\mathrm{tors}$ differs from this line bundle
	only by an exceptional divisor. These divisors pair
	trivially with pullbacks of closed $(n-1,n-1)$-forms.
	Consequently, $c_1(\cG^*)=0$ in the functional sense of
	\cite[Section~2.3]{CDM-flatness}.
	
	By \cite[Theorem~1.2]{CDM-flatness}, $\cG|_X$ is unitary
	flat, and hence belongs to $\Vect^{\wf}(X)$ by
	Proposition~\ref{prop:unitary-uniform-generation}.
	Since $\cA$ is simple in $\Vect^{\wf}(X)$, the inclusion $\cG|_X\hookrightarrow\cA$
	is therefore an isomorphism. Consequently every semisimple object is unitary flat.
\end{proof}

The above proposition implies the following partial purity result:

\begin{Corollary}\label{prop:weak-flat-partial-purity}
	Let $X'$ be a smooth normally big complex variety,
	let $i:X\hookrightarrow X'$ be a big open subset,
	and let $x\in X(\CC)$.
	Then restriction induces an equivalence between the
	categories of semisimple weakly flat bundles on $X'$
	and $X$.
	
	Consequently, the faithfully flat homomorphism
	\(
	\pi_1^S(X,x)\to\pi_1^S(X',x)
	\)
	has a pro-unipotent kernel and it induces an isomorphism on the
	maximal pro-reductive quotients.
\end{Corollary}

\begin{proof}
	Since $X'\setminus X$ has complex codimension at least
	two, we have an isomorphism
	\(
	\pi_1^{\mathrm{top}}(X^{\an},x)
	\simeq
	\pi_1^{\mathrm{top}}((X')^{\an},x).
	\)
	Let $\cE$ be a semisimple weakly flat bundle on $X$.
	By Proposition~\ref{prop:stable-unitary-weak-positivity},
	it is unitary flat.
	Its unitary representation therefore defines a unitary
	flat holomorphic bundle $\cE'_{\an}$ on $(X')^{\an}$
	extending $\cE^{\an}$.
	The sheaf $\cF:=i_*\cE$ is coherent and reflexive.
	Both $\cF^{\an}$ and $\cE'_{\an}$ are reflexive and
	agree on $X^{\an}$, so
	\(
	\cF^{\an}\simeq\cE'_{\an}.
	\)
	Thus $\cF$ is a unitary flat vector bundle on $X'$.
	By Proposition~\ref{prop:stable-unitary-weak-positivity},
	it is semisimple and weakly flat.
	Restriction preserves unitary flat bundles and is fully
	faithful, so we get the asserted equivalence. 
	This gives also the asserted isomorphism on maximal pro-reductive
	quotients.

The homomorphism $i_*:	\pi_1^S(X,x)\to\pi_1^S(X',x)$ is faithfully flat by
	Proposition~\ref{prop:comparison-of-fund-groups}.
	The equivalence just proved shows that every simple
	representation of $\pi_1^S(X,x)$ factors through $\pi_1^S(X',x)$.
	Hence the kernel $K$ of $i_*$ acts trivially on the graded pieces of a
	composition series of every finite-dimensional
	representation of $\pi_1^S(X,x)$.
	In a basis adapted to such a filtration, its image
	consists of upper unitriangular matrices.
	Consequently, $K$ is pro-unipotent.
\end{proof}

\subsection{Chern classes of weakly flat bundles on normal surfaces}
\label{sec:weak-flat-surface-chern}

In this subsection the ground field is $\mathbb C$.

\begin{Lemma}\label{lem:unitary-surface-chern}
	Let $S$ be a normal projective surface, let
	$i:U\hookrightarrow S$ be a smooth big open subset,
	and let $\cE$ be a unitary flat vector bundle on $U$.
	Put $\cF=i_*\cE$. Then
	\[
	c_1(\cF)\equiv0,
	\qquad
	\int_S\ch_2(\cF)=0.
	\]
\end{Lemma}

\begin{proof}
	By Selberg's lemma \cite{Selberg1960}, the image of the unitary monodromy
	representation has a torsion-free subgroup of finite index.
	The corresponding finite \'etale cover $U'\to U$ extends,
	by normalization, to a finite morphism
	$\gamma:S'\to S$ from a normal projective surface.
	Choose a projective log resolution $\pi:T\to S'$
	which is an isomorphism over $U'$.
	 
	By Kashiwara's theorem \cite[Theorem 3.1]{Kashiwara1981}, quasi-unipotence of monodromy at
	infinity is independent of the normal compactification.
Since \(S'\setminus U'\) consists of a finite number of points, quasi-unipotence at infinity is vacuous for this normal compactification. So the monodromies around the	components of $T\setminus U'$ are quasi-unipotent.
	They are unitary, hence have finite order, and belong to
	a torsion-free group. Consequently they are trivial.
	The pulled-back unitary local system therefore extends
	to $T$. Its associated holomorphic bundle is algebraic and it defines a numerically flat vector bundle
	$\cE_T$ on $T$. Since
	\[ 
	(\pi_*\cE_T)^{**}\simeq\gamma^{[*]}\cF,
	\]
Lemma~\ref{lem:nf-birational-chern} implies that  
	\[
c_1(\gamma^{[*]}\cF)\equiv0
	\qquad\hbox{and}\qquad 
\int_{S'}\ch_2(\gamma^{[*]}\cF)=0.
\]	
So the assertion follows from	\cite[Proposition~4.2]{La-Inters}.
\end{proof}

\begin{Proposition}\label{prop:weak-flat-surface-chern}
	Let $S$ be a normal projective surface, let
	$i:U\hookrightarrow S$ be a big open subset,
	and let $\cE$ be a weakly flat vector bundle on $U$.
	Put $\cF=i_*\cE$. Then $\cF$ is slope $H$-semistable
	for every ample divisor $H$ on $S$, and
	\[
	c_1(\cF)\equiv0,
	\qquad
	\int_S\ch_2(\cF)=0.
	\]
\end{Proposition}

\begin{proof}
	Replacing $U$ by $U\cap S_{\mathrm{reg}}$ does not change
	$\cF$, so we may assume that $U$ is smooth.
	By Proposition~\ref{prop:stable-unitary-weak-positivity}
	and the finite-length argument in its proof, $\cE$
	admits a finite filtration by subbundles
	\[
	0=\cE_0\subset\cE_1\subset\cdots\subset\cE_s=\cE
	\]
	whose quotients are unitary flat.
	Set
	\(
	\cF_i:=i_*(\cE_i/\cE_{i-1}).
	\)
	Lemma~\ref{lem:unitary-surface-chern} and additivity
	of the first Chern class give $c_1(\cF)\equiv0$.
	Repeated application of
	Lemma~\ref{lem:surface-chern-subadditivity} gives
	\[
	\int_S\ch_2(\cF)
	\le\sum_{i=1}^s\int_S\ch_2(\cF_i)=0.
	\]
	Slope semistability follows from the curve argument
	in the proof of Theorem~\ref{snf-wp:equivalence}.
	
	For the opposite inequality, fix a very ample line
	bundle $H$ on $S$.
	Use the coordinate root covers from the proof of
	Proposition~\ref{operations:root-cover-criterion},
	and normalize an irreducible component dominating $S$.
	For every integer $m>0$, this gives a finite surjective
	morphism $f_m:S_m\to S$ from a normal projective surface
	and an ample globally generated line bundle $A_m$ such that
	\[
	f_m^*H\simeq A_m^m,
	\qquad
	\cG_m:=f_m^{[*]}\cF\otimes A_m
	\]
	is ample on $f_m^{-1}(U)$.
	Here the ampleness assertion follows by applying the
	necessity argument of that proposition over $U$.
	
	Put $r=\rk\cF$ and assume $r>0$.
	Lemma~\ref{snf-wp:ample-surface-segre}, together with
	$c_1(\cF)\equiv0$ and
	\cite[Proposition~4.2]{La-Inters}, gives
	\[
	0<
	\frac{s_2(\cG_m)}{\deg f_m}
	=
	\int_S\ch_2(\cF)
	+\frac{r(r+1)}{2m^2}H^2.
	\]
	Letting $m$ tend to infinity proves
	$\int_S\ch_2(\cF)\ge0$, and hence the required equality.
\end{proof}

\begin{Corollary}\label{cor:weak-flat-implies-snf-char-zero}
	Let $X$ be a normal quasi-projective variety over an
	algebraically closed field $k$ of arbitrary characteristic.
	Then every weakly flat vector bundle on $X$ is strongly
	numerically flat.
\end{Corollary}

\begin{proof}
	Let $\cE$ be a weakly flat vector bundle on $X$.
	Applying Lemma~\ref{snf-wp:elementary-properties}(3)
	to $\cE$ and $\cE^*$ shows that $\cE$ is numerically flat.	
	Let $g:S\to X$ be a morphism from a normally big normal
	surface, and let $i:S\hookrightarrow\overline S$ be its
	small compactification.
	By Lemma~\ref{snf-wp:elementary-properties}(1),
	the bundle $g^*\cE$ is weakly flat.
	
	If $\operatorname{char}k>0$, choose an ample line bundle
	on $\overline S$. Applying the implication
	$(2)\Rightarrow(3)$ of Theorem~\ref{snf-wp:equivalence}
	to $g^*\cE$ gives
	\[
	\int_{\overline S}\ch_2(i_*g^*\cE)=0.
	\]
	
	If $\operatorname{char}k=0$, standard reduction
	arguments allow us to assume that $k=\CC$.
	Apply Proposition~\ref{prop:weak-flat-surface-chern}
	to $(g^*\cE)|_{S_{\mathrm{reg}}}$.
	Removing the singular points of $S$ does not change
	its reflexive extension to $\overline S$, so we obtain
	the same equality.
	
	Thus $\cE$ satisfies the defining surface condition
	for strong numerical flatness.
\end{proof}

\medskip

\section*{Declaration on the use of generative AI}

During the development and preparation of this article, the author
used ChatGPT (OpenAI; GPT-6 Astra) extensively as an interactive
tool to explore proof strategies for some of the results and to
assist with drafting and revision. The author independently
checked all mathematical arguments and references, rewrote the
generated material, and takes full responsibility for the
contents of the article.

\section*{Acknowledgements} 

The author would like to thank Mihai Fulger for providing Example \ref{Fulger}.

The author was partially supported by the National Science Centre, Poland, contract number 2025/59/B/ST1/02168.

\bibliographystyle{amsalpha}
\bibliography{References}

\end{document}